\documentclass[10pt,a4paper]{article}
\usepackage{newunicodechar}
\newunicodechar{̈}{\"{}}
\usepackage[backend=biber,
    style=alphabetic, sorting=nyt,
    maxalphanames=10, maxbibnames=10]{biblatex}
\usepackage{amsmath}
\usepackage{amssymb}
\usepackage{amsthm}
\usepackage{authblk}
\usepackage{booktabs}
\usepackage{dcolumn}
\usepackage{fancyhdr}
\usepackage{geometry}
\usepackage{graphicx}
\usepackage{latexsym}
\usepackage{xcolor}
\usepackage{hyperref}
\hypersetup{
  colorlinks,
  citecolor=purple,
  linkcolor=blue,
  urlcolor=green}

\usepackage{graphicx}
\usepackage{siunitx}
\usepackage{tipa}
\usepackage{stackengine}
\usepackage{mathrsfs}
\usepackage{enumitem}
\usepackage{breqn}
\usepackage{dirtytalk}
\usepackage{pxfonts}
\usepackage{lineno}
\usepackage{comment}
\newtheorem{theorem}{Theorem}[section]
\newtheorem{lemma}[theorem]{Lemma}
\newtheorem{proposition}[theorem]{Proposition}
\newtheorem{corollary}[theorem]{Corollary}
\newtheorem*{repeatedtheorem}{\repeatedtheoremname}
\newcommand{\repeatedtheoremname}{}

\numberwithin{equation}{section}

\theoremstyle{definition}
\newtheorem{remark}[theorem]{Remark}
\newtheorem{definition}{Definition}[section]
\newtheorem{example}[theorem]{Example}

\newcommand{\F}{\mathscr{F}}
\newcommand{\R}{\mathbb{R}}
\newcommand{\vart}{\vartheta}
\newcommand{\eps}{\varepsilon}

\title{$\lambda$ -- geodesic convexity of the squared\\ Sliced Wasserstein Distance}

\author{Giacomo~Cozzi\thanks{\scriptsize Dipartimento di Matematica ``Tullio Levi-Civita'', Università degli Studi di Padova, Via Trieste, 63, 35131 Padova PD, Italy. {\tt cozzi@math.unipd.it}},  Filippo~Santambrogio\thanks{\scriptsize 
Universit\'e Lyon 1, Ecole Centrale de Lyon, INSA Lyon, Universit\'e Jean Monnet, CNRS, ICJ, UMR 5208, Villeurbanne, France,
\texttt{santambrogio@math.univ-lyon1.fr}}}

\begin{document}
\maketitle
\begin{abstract}
{The squared Wasserstein distance with respect to a fixed target measure is not a convex function along Wasserstein geodesics over $\R^d$ when $d\geq 2$. Still, in dimension $d=1$, this function is $2$-geodesically convex. We prove that the square of the sliced Wasserstein distance (which is a distance constructed averaging one-dimensional Wasserstein distances) is $\lambda$-geodesically convex (possibly with $\lambda \leq 0$) on sets of measures satisfying upper bounds on their density, when the target measure satisfies some given geometric conditions. We also present some consequences of this result in the study of the sliced Wasserstein flow.
}
\end{abstract}

\tableofcontents

\section{Introduction and main results}
Optimal transport is one of the richest fields in analysis, thanks to which many remarkable results in different areas of mathematics were achieved in the last decades.
The Monge problem, first formulated in his \say{Mémoire sur la Théorie des Déblais et des Remblais}, aimed at minimizing the average displacement needed for moving a certain distribution of mass (source) to a different prescribed configuration (target): given two probability densities $\rho, \nu \in \mathcal{P}(\R^d)$, a Borel map $T\colon \R^d \to \R^d$ is said to be a \textit{transport map} from $\rho$ to $\nu$ if the nonlinear constraint \begin{equation*}
    \nu(A)=\rho(T^{-1}(A))
\end{equation*}
is satisfied for any Borel set $A$. In this case we write $T_{\#}\rho=\nu$, and we say $T$ is the \say{optimal transport map} for the Monge problem with source $\rho$ and target $\nu$, provided
\begin{equation*}
    \int_{\R^d}c(T(x),x)\,d\rho(x)=\inf_{S_{\#}\rho=\nu}\int_{\R^d}c(S(x),x)\,d\rho(x),
\end{equation*}
where $c(x,y)$ is a positive cost function. When $c(x,y)=|x-y|^p$, with $p\geq 1$, the $p-$th root of the infimum above defines a distance between probability measures, called the $p$-\textit{Wasserstein distance}, and denoted by $W_p$.

Wasserstein distances enjoy many properties that make them a useful tool in theoretical research, as they metrize the space of probability measures with finite $p$-moments and provide it of a geodesic space structure. At the same time, they are widely used in applications, being one of the most common choices as a cost for machine learning. However, computing Wasserstein distances in dimensions $d\geq 2$ can be a computationally expensive task. This is the reason why many alternatives to the Wasserstein distance have been introduced in the literature. 

One of them is the following: let $\rho$ and $\nu$ be probability measures in $\mathbb{R}^d$, both with finite second moments. 
For any unit direction $\vartheta \in \mathbb{S}^{d-1}$ we can define $\rho^\vartheta=(\pi^\vartheta)_{\#}\rho$ and $\nu^\vartheta=(\pi^{\vartheta})_{\#}\nu$, where $\pi^{\vartheta} \colon \mathbb{R}^{d} \to \mathbb{R}$ is the canonical projection on the line $\vartheta \mathbb{R}$ and the symbol $\#$ indicates the push-forward of the measure. 
We define 
\begin{equation}
\label{SW distance}
    SW_2(\rho,\nu)\coloneqq \left(\fint W^2_2(\rho^\vartheta,\nu^\vartheta)\,d\sigma(\vartheta)\right)^{\frac{1}{2}}.
\end{equation}
Here $\fint$ stands for the averaged integral, i.e., $\fint_E dm=\frac{1}{m(E)} \int_E dm$, and $d\sigma$ is the surface measure on the sphere. We reserve $\sigma$ for this $(d-1)$-dimensional measure and write $\mathcal H^k$ for the $k$-dimensional Hausdorff measure on other integration domains.
It is well known that \eqref{SW distance} defines a distance over $\mathcal{P}_2(\R^d)$, the space of probability measures with finite second moments on $\R^d$, called \textit{sliced Wasserstein distance}: this is comparable with the usual Wasserstein distance $W_2$, but it presents many advantages for applications, since it relies only on one-dimensional transport constructions \cite{kolouri2019generalized, Nadjahi2020StatisticalAT, ParSle, rustamov2023intrinsic}. 

Fix a target measure $\nu \in \mathcal{P}_2(\R^d)$. The main result of the present paper concerns the convexity of the functional 
\begin{equation*}
    \mathscr{F}_\nu(\rho)=\frac{1}{2}SW_2^2(\rho,\nu)
\end{equation*}
along Wasserstein geodesics, under certain assumptions on $\nu$. 

Here is an informal statement of the main result, which will then be presented as Theorem \ref{Main theorem}:
\renewcommand{\repeatedtheoremname}{Theorem \ref{Main theorem}}
\begin{repeatedtheorem}
    Let $M>0$, $K=\{\rho \in \mathscr{P}_2(\R^d)\colon \|\rho\|_{L^\infty}\leq M\}$, and $\nu \in \mathscr{P}_2(\R^d)$. 
    Assume that $\nu$ satisfies the technical assumption \eqref{assumption nu}. 
%
For any $\rho_0,\rho_1 \in K$, let $\rho_t$ be the constant-speed Wasserstein geodesic connecting them. Then, it holds
\begin{equation*}
    \mathscr{F}_\nu(\rho_t)\leq(1-t)\mathscr{F}_\nu(\rho_0)+t\mathscr{F}_\nu(\rho_1)-\frac{\lambda}{2}t(1-t)W^2_2(\rho_0,\rho_1),
\end{equation*}
with $$\lambda=\frac{1}{d}-cM^{\frac{1}{d}},$$
where $c$ is a constant only depending on $\nu$ and not on $\rho_0,\rho_1,\in K$.
\end{repeatedtheorem}
This is in contrast with what happens for the Wasserstein distance $W_2$ which, rather, is in general not expected to be convex along geodesics when $d \geq 2$ \cite{ambrosio2005gradient}.
We stress here that, in general, $\lambda$ is not expected to be positive and the convexity degenerates (i.e., $\lambda$ goes to $-\infty$) as the $L^\infty$ bound $M$ for the set $K$ increases, as well as when we change the constants appearing in the precise formulation of the assumptions needed on $\nu$, namely in equation \eqref{assumption nu}. 

An interesting application of this result concerns the study of the 
\textit{sliced Wasserstein flow} (SWF), which
is the isotropic and continuous-in-time version of the \textit{Iterative Distribution Transfer} (IDT), first proposed in \cite{pitie2007automated}. This is an algorithm based on a sequence of one-dimensional optimal matchings between projections of two given measures along different axes, chosen randomly at each time step. 

More precisely: fix a time step $\tau>0$. Starting from a probability density $\rho_0$, choose randomly an orthonormal basis $\{e_{1},\ldots,e_{d}\}$, and consider the projections $(\pi^{e_1}_{\#}\rho_0,\ldots,\pi^{e_d}_{\#}\rho_0)=:(\rho_0^{1},\ldots,\rho_0^{d})$. Let $(T^1,\ldots,T^d)$ be the optimal transport map (for the quadratic cost) between $(\rho_0^1,\ldots,\rho_0^d)$ and $(\pi^{e_1}_{\#}\nu,\ldots,\pi^{e_d}_{\#}\nu)=:(\nu_1,\ldots,\nu_d)$. We now want to move each point $x\in \text{spt }\rho_0$ in the direction $e_i$ of a displacement given by the displacement of the map $T^{i}$ at the point $x \cdot e_i$, and combine these displacements for all the elements of the randomly chosen basis. To this aim, consider $v^{IDT}(x):=\sum_{i=1}^d (T^i(x\cdot e_i)-x\cdot e_i)e_i$ and define $\rho^\tau_1 =(\mathrm{id}+\tau v)_{\#}\rho_0$. Heuristically, one expects $\rho^\tau_1$ to be \say{closer} than $\rho_0$ to $\nu$. Finally, we iterate the construction, obtaining a sequence $\{\rho^\tau_{k}\}_{k \ge 0}$. When $k\to +\infty$, one expects (or hopes) to recover the target probability $\nu$.

Starting from this, Bernot proposed the sliced Wasserstein flow operating the following two modifications: first, instead of choosing a random orthonormal basis, at each iteration, we consider the collection of all projections $\{\rho^{\vart}_0\}_{\vart \in \mathbb{S}^{d-1}}$ and  $\{\nu^{\vart}\}_{\vart \in \mathbb{S}^{d-1}}$. Denoting by $T^\vart$ the optimal transport map between $\rho_0^\vart$ and $\nu^\vart$, we define the isotropic version of the velocity field $v^{IDT}$. Namely, we compute
\begin{equation}
\label{velocity field}
v[\rho_0](x):= \fint_{\mathbb{S}^{d-1}} (T^\vart(x\cdot \vart)-x\cdot \vart)\vart\, d\sigma(\vart),
\end{equation}and define $\rho^\tau_1:= (\mathrm{id}+\tau v(x))_{\#}\rho_0$. We iterate the construction obtaining a sequence $\{\rho^\tau_k\}_{k\ge 0}$. The second difference from the IDT algorithm consists in considering the continuous-in-time analog of this construction, that is, we take the limit as $\tau \to 0^+$ of the aforementioned sequence. Thus, we obtain an absolutely continuous curve $(\rho_t)_{t\geq 0}$ that solves 
\begin{equation}
\label{SWF eq rho_0}
    \begin{cases}
        \partial_t \rho +\nabla \cdot (\rho \,v[\rho])=0,\\
        \rho(0)=\rho_0,
    \end{cases}
\end{equation}

This flow is expected to be an evolution which lets $\rho_t$ approach the target $\nu$, but it only relies on one-dimensional constructions, avoiding the growth of computational costs when the dimension $d$ of the space is large (curse of dimensionality). 

In Section \ref{S: Sliced wasserstein distance and sliced wasserstein flow}  we will show that the sliced Wasserstein flow with target measure $\nu$ corresponds to the Wasserstein gradient descent for the functional $\F_\nu$. This fact is relevant since it allows to apply the theory of Wasserstein gradient flows to the study of the SWF. In particular, the existence of solutions to the SWF was proved by Bonnotte in his PhD thesis \cite{bonnotte2013unidimensional} by JKO methods \cite{SantSurvey}. In the same work he also proved bounds for the $L^p $ norm along the flow (see Theorem \ref{stime Bonnotte Lp} below) and addressed some questions regarding the long time asymptotics and uniqueness of the flow. 
Regarding the first of these two questions, we can mention \cite{cozzi2025long}, where the authors proved that, if $\nu$ is the Gaussian measure, the flow converges to $\nu$. Other results in this directions have been obtained in \cite{vauthier2025towards}. 
For what concerns uniqueness of solutions, as a byproduct of Theorem \ref{Main theorem}, in this paper we are able to obtain the following
\begin{corollary}
\label{uniqueness corollary}
    If $\nu \in \mathscr{P}_2$ satisfies \eqref{assumption nu} and $\rho_0 \in \mathscr{P}_2(\R^d)\cap L^\infty(\R^d)$, then the sliced Wasserstein flow is well-posed in $L^\infty(\R^d)$, i.e., there exists a unique solution $\{\rho_t\}_{t\geq 0}$ such that $t\mapsto ||\rho_t||_{L^\infty}$ is locally bounded. 
\end{corollary}

This is just an example of the consequences of our $\lambda-$geodesic convexity result. Among others, we cite the fact that it also implies the convergence at a rate $\tau$ in $W_2$ distance of the corresponding JKO scheme to the continuous solution (from the theory developed in \cite{ambrosio2005gradient}) and the fact that the functional $SW_2^2$ can then be used in the flow interchange technique (see \cite{matmccsav}), but we do not develop these directions in details.

The paper is structured as follows: in Section \ref{preliminaries}, we report the main results needed from the theory of optimal transport, as well as some results related to the sliced Wasserstein distance and the sliced Wasserstein flow. In Section \ref{S: from Lip to Conv}, we will show how to obtain the geodesic convexity of $\mathscr{F}_\nu$, once Lipschitz bounds on the velocity field $v$ are known. In Section \ref{S: Lip bounds}, then, we will derive these very Lipschitz bounds, first on the class of probability measures with bounds in $L^p$ under very strong assumptions on $\nu$. Then, we will state and comment the main result, Theorem \ref{Main theorem}, which requires more natural assumptions on $\nu$ but holds under $L^\infty$ bounds: its proof will be presented in Section \ref{sec: main theorem}. Finally, in the last Section we will prove Corollary \ref{uniqueness corollary} and, more generally, discuss the uniqueness for the Sliced Wasserstein Flow.
\section{Preliminaries}
\label{preliminaries}
\subsection{Notation}
For a function $u : \mathbb{R}^d \mapsto \mathbb{R}$ smooth enough, we denote by $\nabla u$ its
gradient and $D^2 u$ its Hessian matrix. Moreover, for a map $T : \mathbb{R}^d \mapsto \mathbb{R}^d$, its Jacobian matrix is denoted by $DT$. We denote by \say{$\mathrm{id}$} the identity map. When referring to the target measure $\nu$, we will denote by $N_\vart (r)=\nu^\vart((-\infty,r])$, where $\nu^\vart=\pi^\vart_{\#}\nu$. Similarly, for a generic distribution $\rho$, we use the notation $R_\vart (r)=\rho^\vart((-\infty,r])$. We will often make an abuse of notation using $\rho,\nu$ as a symbol both for the probability measures or their densities, when no confusion can arise. We will denote by $B^{(d)}_r(x)$ the $d$-dimensional Euclidean ball centered at $x$ with radius $r$. When there is no ambiguity, we will write $B_r(x)$, and possibly just $B_r$ when $x=0$. Throughout the paper \say{$\text{spt }$} stands for the support of a measure, i.e., the smallest closed set on which it is concentrated. Finally, for any $p\geq 1$, we denote by $p'$ the H\"older conjugate exponent: $1/p+1/p'=1$.
We denote by $\omega_d=\mathcal L^d(B_1(0))$ the volume of the unit ball in $\R^d$; thus $\sigma(\mathbb S^{d-1})=d\omega_d$.
\subsection{Preliminary results in optimal transport}
Given any open subset $\Omega \subset \mathbb{R}^d$, we denote by $\mathscr{P}(\Omega)$ the set of probability measures with support contained in $\Omega$. For $p\geq 1$, we denote by $\mathscr{P}_p(\Omega)$ the set of probability measures on $\Omega$ with finite moment of order $p$ (in case $\Omega$ is unbounded), so that we have $\mathscr{P}_p(\Omega)=\{\mu\in \mathscr{P}(\Omega)\,:\, \int|x|^pd\mu<+\infty\}$.

Given $\mu, \nu \in \mathscr{P}(\Omega)$, we say that $T$ is a transport map from $\mu$ to $\nu$, and that $\nu$ is the push-forward of $\mu$ via $T$ (and we write $\nu=T_{\#}\mu$), if for every Borel set $A \subset \Omega$ we have
\[\nu(A)=\mu(T^{-1}(A))\]
or, equivalently, if for any Borel map $\phi \colon \Omega \to \mathbb{R}$ it holds
\[\int_{\Omega} \phi(x) d\nu(x) =\int_{\Omega}\phi(T(x))d\mu(x).\]
\begin{definition}
    \label{costs}
    Let $\mu, \nu \in \mathscr{P}\left(\Omega\right)$ and let $c:\Omega\times\Omega\to\mathbb R$ be a continuous cost function . 
    The map $T : \mathbb{R}^d \to \mathbb{R}^d$ is said to be an optimal transport map for the cost $c$ if it solves the problem
    \begin{equation}
        \label{Monge_pb}
        \min_{S} \left\{\int_{\Omega}c(x,S(x))d\mu(x), \hspace{2mm} S \hspace{2mm} \text{Borel map such that} \hspace{2mm} S_{\#}\mu=\nu \right\}.
    \end{equation}
    The map $\varphi: \mathbb{R}^d \to \mathbb{R}$ is said to be a Kantorovich potential if it solves the problem
    \begin{equation}
    \label{Potential_pb}
        \max_{\psi}\left\{\int_{\Omega}\psi(x)d\mu(x)+\int_{\Omega}\psi^c(x)d\nu(x)\right\},
    \end{equation}
    where $\psi^c(x) \coloneqq \inf_{y \in \mathbb{R}^d} \left\{c(x,y)-\psi(y)\right\}$.

Let $\Pi_\mu^\nu\coloneqq \{\gamma\ \in \mathcal{P}(\Omega\times\Omega)\colon p^1_{\#}\gamma=\mu, p^2_{\#}\gamma=\nu\}$, where $p^i (x_1,x_2)=x_i$. We define the $p$-Wasserstein distance from $\mu$ to $\nu$ to be
\begin{equation*}
    W_p(\mu,\nu)=\left(\int_{\Omega\times \Omega} |x_1-x_2|^{p}\,d\pi(x_1,x_2)\right)^{\frac{1}{p}}=\inf_{\gamma \in \Pi_\mu^\nu}\left(\int_{\Omega\times\Omega}|x_1-x_2|^p\,d\gamma(x_1,x_2)\right)^{\frac{1}{p}}.
\end{equation*}
     It is well known that $W_p$ defines a distance over $\mathscr{P}_p(\Omega)$ for any $p\geq 1$ (see, for example, \cite[Chapter 5]{santambrogio2015optimal}).
\end{definition}
\begin{remark} The problems \eqref{Monge_pb} and \eqref{Potential_pb} do not always admit a solution. However, when they exist, the two quantities defined by \eqref{Monge_pb} with \say{$\min$} replaced by \say{$\inf$}, and \eqref{Potential_pb} with \say{$\max$} replaced by \say{$\sup$} coincide. Moreover, when a solution $T$ to the Monge problem exists, then the joint density $\pi$ appearing in the definition of $W_p$ is the measure concentrated on $\text{gr}(T)\subset \Omega \times \Omega$, where $\text{gr}T$ denotes the graph of $T$, namely $\text{gr}T=\{(x,T(x))\in \Omega\times\Omega\colon x\in \Omega\}$. As a consequence, if $T$ is an optimal transport map and $\varphi$ is the Kantorovich potential for the associated problem, it holds
\begin{equation*}
\label{Wassersein distance, optimal transport, kantorovich potential}
    W_p^p(\mu,\nu)=\int_{\Omega} |T(x)-x|^p\,d\mu=\int_{\Omega}\varphi(x)\,d\mu(x)+\int_{\Omega}\varphi^c(x)\,d\nu(x).
\end{equation*}
\end{remark}
We are interested in the case $\Omega=\mathbb R^d$ and $c(x,y)=\frac 12|x-y|^2$. In this setting we have a relation between the optimal $T$ and the optimal $\varphi$ (which is differentiable almost everywhere
with respect to the source measure, when it is absolutely continuous) which is
\begin{equation}
\label{KantVsTransp}
T(x)=x-\nabla \varphi(x).
\end{equation}
In particular, we have $T=\nabla u$ where $u(x)=\frac12|x|^2-\varphi(x)$, and one can check that $u$ is convex. This is summarized in the following celebrated theorem.

\begin{theorem}[Brenier's theorem, see \cite{brenier1987decomposition}]
\label{Brenier's_thm}
    Let $\mu$, $\nu$ be two absolutely continuous probability measures in $\mathscr{P}_2(\mathbb R^d)$. A map $T \colon \mathbb{R}^d \to \mathbb{R}^d$ such that $T_{\#}\mu=\nu$ is optimal for the optimal transport problem \eqref{Monge_pb} with the cost $c(x,y)=\frac 12|x-y|^2$ if and only if there exists a convex function $u \colon \mathbb{R}^d \to \mathbb{R}$ such that $T=\nabla u$ a.e. 
    Moreover if this is the case, then $u$ must satisfy the Monge-Amp\`ere equation
    \begin{equation*}
    \nu(\nabla u(x))    \det D^2 u(x)=\mu(x)
    \end{equation*}
    a.e. on $\mathbb{R}^d$.
\end{theorem}
For the one-dimensional problem, we are able to characterize optimal transport maps. A detailed analysis of the one-dimensional case is available in \cite{santambrogio2015optimal}, Chapter 2.
\begin{theorem}[Brenier's theorem in $d=1$]
\label{brenier 1d}
    Let $\mu$, $\nu$ be two probability densities with finite second moments. Then there exists a unique non-decreasing transport map from $\mu$ to $\nu$, which is optimal and given by
    \[T(x)=F_\nu^{[-1]}\left(F_{\mu}(x)\right),\]
    where $F_{\mu}(x)=\mu((-\infty,x])$ and $F^{[-1]}(y)=\inf\{t\in \mathbb{R} \colon F(t)\geq y\}$.
\end{theorem}
The metric space $(\mathscr{P}_2(\R^d),W_2)$ is a geodesic space, namely: given any two elements $\rho_0,\rho_1 \in \mathscr{P}_2(\R^d)$, there exists a curve $\rho_t\colon[0,1]\to \mathscr{P}_2(\R^d)$ connecting them with the property that 
\begin{equation}
\label{constant speed geo}
    W_2(\rho_s,\rho_t)=|t-s|W_2(\rho_0,\rho_1) \hspace{3mm}\text{ for any } 0<s<t<1.
\end{equation}
\begin{definition}
    A curve satisfying \eqref{constant speed geo} is called \textit{constant speed geodesic} from $\rho_0$ to $\rho_1$.
    
    A functional $\mathscr{F}\colon \mathscr{P}_2(\R^d)\to \R$ is $\lambda$-geodesically convex if for every two points $\rho_0, \rho_1 \in \mathscr{P}_2(\R^d)$ there exists a constant speed geodesic $\rho(t)$ connecting $\rho(0)=\rho_0$ and $\rho(1)=\rho_1$ such that
\begin{equation}
\label{lambda convexity def}
    \mathscr{F}(\rho_t)\leq (1-t)\mathscr{F}(\rho_0)+t\mathscr{F}(\rho_1)-\frac{\lambda}{2}t(1-t)W_2^2(\rho_0,\rho_1).
\end{equation}
\end{definition}
The case in which $\lambda>0$ is generally referred to as \textit{strong convexity}, whereas for $\lambda=0$ we recover the usual definition of convexity. If $t\mapsto \mathscr{F}(\rho_t)$ is twice differentiable, \eqref{lambda convexity def} is equivalent to
\begin{equation*}
    \frac{d^2}{dt^2}\mathscr{F}(\rho_t)\geq \lambda W_2^2(\rho_0,\rho_1).
\end{equation*}
The convexity of a functional over the space of probability measures, though not being essential for proving the well posedness of the associated gradient flow in the Wasserstein space, is the typical assumption under which uniqueness and contractivity results may be obtained (cf. \cite{ambrosio2005gradient, SantSurvey}).

Throughout this paper, we will make use several times of the following two results.
\begin{theorem}{\cite[Theorem 5.27]{santambrogio2015optimal}}
\label{geodetiche}
Let $\rho_0,\rho_1 \in \mathscr{P}_2(\R^d)$ with $\rho_0$ absolutely continuous
with respect to the Lebesgue measure, and suppose there exists an optimal transport map $T$ for the quadratic cost $c(x,y)=|x-y|^2$, i.e., a solution of \eqref{Monge_pb}. Then, the unique geodesic $\rho_t \colon [0,1]\to \mathscr{P}_2(\R^d)$ from $\rho(0)=\rho_0$ to $\rho(1)=\rho_1$ is given by
\begin{equation*}
    \rho_t=(T_t)_{\#}\rho_0,
\end{equation*}
where
\begin{equation*}
    T_t(x)=(1-t)x+tT(x)
\end{equation*}
for any $x \in \text{spt }\,\rho_0$.
Then $T_t$ is injective for $t \in [0,1)$ and the curve $(\rho_t)_{t\geq 0}$ solves
\begin{equation*}
    \partial_t \rho_t +\nabla \cdot (\rho w )=0,
\end{equation*}
where $w_t(x)=(T-\mathrm{id})\circ T_t^{-1}(x)$.
\end{theorem}

\subsection{Sliced-Wasserstein distance and sliced-Wasserstein flow}
\label{S: Sliced wasserstein distance and sliced wasserstein flow}
Given two probability measures $\mu,\nu \in \mathscr{P}_2(\mathbb{R}^d)$, we define their sliced Wasserstein distance as in \eqref{SW distance}. As said in the introduction, the sliced Wasserstein distance does not coincide with the classical $W_2$. However, the two distances generate the same topology (see \cite{bayraktar2021strong}). Moreover, whenever $\text{spt }\mu,\text{spt }\nu\subset B(0,R)$, there exist an exponent $\alpha>0$ and constant $C(R)>0$ such that it holds $C(R)W_2(\mu,\nu)\leq SW_2^\alpha(\mu,\nu)\leq W^\alpha_2(\mu,\nu)$  (see \cite[Chapter 5]{bonnotte2013unidimensional} and \cite{CarFigMerWan}). Many properties, which highlight surprising differences in the behavior of $SW_2$ compared to $W_2$, have been studied in  \cite{ParSle}. We also mention the possibility of defining a $p$-sliced distance $SW_p$, via \begin{equation*}
SW_p \coloneqq\left(\fint_{\mathbb{S}^{d-1}}W_p^p((\pi^{\vartheta})_{\#}\mu,(\pi^{\vartheta})_{\#}\nu)d\sigma(\vartheta)\right)^{1/p},\end{equation*} which is sometimes more convenient for sharp comparisons with the standard Wasserstein distances (see \cite{CarFigMerWan}). Also, a distance of the form
\begin{equation*}
    SW_{p,q}(\mu,\nu)\coloneqq \left(\fint_{\mathbb{S}^{d-1}}W^{q}_p(\mu^\vart,\nu^\vart)\right)^{\frac{1}{q}}
\end{equation*}
can be introduced, with a particular interest when $q=+\infty$, in which case one wants to maximal Wasserstein distance among projections.

In this Section we want also to show that the velocity field associated with the SWF corresponds to the velocity field of the gradient descent for the functional $\F_\nu$.
Namely, we claim that $v[\rho_t](x)$ as defined in \eqref{velocity field} coincides with the negative gradient of the first variation of $\F_\nu$:
\begin{equation*}
    -\nabla \frac{\delta \F_\nu}{\delta \rho}(\rho)=\fint_{\mathbb{S}^{d-1}}(T_\vart(x\cdot \vart)-x\cdot \vart)\vart \,d\sigma(\vart)\, \text{ for any }x\in \text{spt}\rho.
\end{equation*}
Indeed, since the first variation of the functional $\rho \mapsto \frac12 W_2^2(\rho,\nu)$ is the Kantorovich potential for the transport problem from $\mu$ to $\nu$ (see \cite{santambrogio2015optimal}, proposition 7.17), we get
\begin{equation*}
    \frac{\delta \mathscr{F}_\nu}{\delta \rho}(\rho)(x)=\fint_{\mathbb{S}^{d-1}} \frac{1}{2} \frac{\delta W_2^2(\rho^\vartheta,\nu^\vartheta)}{\delta \rho}(x) d\sigma(\vartheta)=\fint_{\mathbb{S}^{d-1}} \varphi^{\vartheta}(x\cdot \vartheta)d\sigma(\vartheta),
\end{equation*}
where we denoted by $\varphi^{\vartheta}$ the Kantorovich 
potential between $\rho^\vartheta \coloneqq (\pi^\vartheta)_{\#}\rho$ and
$\nu^\vartheta\coloneqq (\pi^\vartheta)_{\#}\nu$.  
Using also \eqref{KantVsTransp}, this in turn leads to
\begin{equation*}
    -\nabla \frac{\delta \mathscr{F}_\nu}{\delta \rho}(\rho)=-\fint_{\mathbb{S}^{d-1}} (\varphi^\vartheta)'(x\cdot \vartheta)\vartheta d\sigma(\vartheta)=\fint_{\mathbb{S}^{d-1}} \left(T^\vartheta(x\cdot \vartheta)-x\cdot\vartheta\right)\vartheta d\sigma(\vartheta),
\end{equation*}
which is precisely the velocity field defined in \eqref{velocity field}. Therefore the SWF with target $\nu$ and source $\rho_0$ corresponds to a solution of \eqref{SWF eq rho_0}

We conclude this section recalling the main result on the Sliced Wasserstein Flow, obtained by viewing it as the gradient flow of $SW_2^2$, which provides existence together with $L^p$ estimates for the solutions:
\begin{theorem}{\cite[Theorem 5.5.3]{bonnotte2013unidimensional}}
    Let $\nu\in \mathscr{P}_2(\R^d)$, $p \in [1,+\infty]$. Given an absolutely continuous probability measure $\rho_0\in \mathscr{P}_2(\R^d)$, there exists a continuous curve $(\rho_t)_{t\geq 0}$ solution of \eqref{SWF eq rho_0}. Moreover, for each time $t\geq 0$, the measure $\rho_t$ has a density in $L^p$ and
    \begin{equation}
    \label{stime Bonnotte Lp}
    \|\rho_t\|_{L^p}\leq e^{t\sqrt{d}/p'}\|\rho_0\|_{L^p}.
    \end{equation}
\end{theorem}
\begin{remark} In \cite{bonnotte2013unidimensional}, the above result is proved under the assumption that $\nu$ is strictly positive and supported in $\overline{B}(0,1)$. However, since the estimate obtained does not depend on this condition, by approximation we can deduce the theorem in the form in which we presented it.
\end{remark}
\section{Lipschitz bounds on $v$ imply $\lambda$-geodesic convexity of $\mathscr{F}_\nu$}
\label{S: from Lip to Conv}
In this section, we show how to derive the geodesic convexity of $\mathscr{F}_\nu$ when the velocity field $v$ is Lipschitz continuous in space. Lipschitz bounds on $v$ will in turn be obtained in Section \ref{S: Lip bounds}.
\begin{lemma}
\label{Lemma 1}
    Let $v[\rho]$ be defined by \eqref{velocity field}. Assume that $K\subset \mathscr{P}_2(\R^d)$ is a geodesically convex set and that for any $\rho \in K$
    \begin{equation*}
        \text{Lip}\left(v[\rho]+\frac{\mathrm{id}}{d}\right)\leq L.
    \end{equation*}
    Then $\rho\mapsto \mathscr{F}_\nu(\rho)$ is $\lambda$-geodesically convex on $K$ for $\lambda=\frac{1}{d}-L$.
\end{lemma}
\begin{proof}Let $\rho_0,\rho_1$ be two probability densities in $K$, and let $(\rho_t)_{t\geq 0}\subset K$ be the Wasserstein geodesic connecting them. Let $T$ be the optimal transport map from $\rho_0$ to $\rho_1$ and $T_t\coloneqq (1-t)\mathrm{id}+tT$. Then, by Theorem \ref{geodetiche} $\rho_t$ solves \begin{equation*}
    \partial_t \rho_t +\text{div}(w_t\rho_t)=0,
\end{equation*}
where $w_t(x)\coloneqq (T-\mathrm{id})\circ T_t^{-1}(x)$ . We have
\begin{equation*}
    \begin{split}
        \frac{d}{dt}\F_\nu(\rho_t)=& \int \frac{\delta \F_\nu}{\delta \rho}(\rho_t)\partial_t \rho_t=\int \nabla \frac{\delta\F_\nu}{\delta\rho}(\rho_t)\cdot w_t \rho_t = -\int v[\rho_t]\cdot w_t \rho_t,
    \end{split}
\end{equation*}
\begin{equation}
\label{derivata seconda}
    \begin{split}
        \frac{d^2}{dt^2}\F_\nu(\rho_t)=&-\underbrace{\int \partial_t v[\rho_t]\cdot w_t \rho_t}_{\coloneqq I_1} - \underbrace{\int v[\rho_t]\cdot \partial_t(w_t\rho_t)}_{\coloneqq I_2}.
    \end{split}
\end{equation}
The term $I_2$ can be estimated as follows; let $\overline{v}$ be a time independent vector field, with the property that 
\begin{equation*}
    \left\|D\left(\overline{v}(x)+\frac{\mathrm{id}}{d}\right)\right\|_{L^\infty}\leq L.
\end{equation*}
Then,
\begin{equation*}
\label{derivata dist secondo termine}
    \begin{split}
        &\frac{d}{dt}\int \overline{v}\cdot w_t \rho_t=\frac{d}{dt}\int \overline{v}\cdot (T-\mathrm{id})(T_t^{-1}(x))T_{t \#}\rho_0\\ &=\frac{d}{dt}\int \overline{v}(T_t(x))\cdot (T(x)-x)\rho_0
        = \int (T(x)-x)\cdot D \overline v (T_t(x))\,  (T(x)-x)\rho_0\\
        &=\int (T(x)-x)\cdot D\left(\overline{v}+\frac{\mathrm{id}}{d}\right)(T_t(x))(T(x)-x)\,\rho_0-\frac{1}{d}\int |T(x)-x|^2\rho_0\\
        &\leq \left(L-\frac{1}{d}\right)\int |T(x)-x|^2\,\rho_0. 
    \end{split}
\end{equation*}
Taking $\overline v=v[\rho_t]$, we obtain
\begin{equation}
\label{primo termine}
    \begin{split}
        -I_2=-\int v[\rho_t]\cdot \partial_t(w_t \rho_t) \geq \left(\frac{1}{d}-L\right)W^2_2(\rho_0,\rho_1).
    \end{split}
\end{equation}
On the other hand, we claim that $I_1$ is negative. Indeed,
\begin{equation*}
    \begin{split}
        I_1=&\int \partial_t(v[\rho_t])\cdot w_t\rho_t= \int \partial_t\left(\fint_{\mathbb{S}^{d-1}}  T_{t,\vartheta}(x\cdot\vartheta)\vartheta \,d\sigma(\vartheta)\right)\cdot w_t \rho_t\,dx\\
        =&\int\partial_t\left[\fint_{\mathbb{S}^{d-1}}    N_\vartheta^{-1}(R_{t,\vartheta}(x\cdot \vartheta))\vartheta \,d\sigma(\vartheta)  \right] \cdot w_t\rho_t\,dx\\
        =& \int \fint_{\mathbb{S}^{d-1}}  \frac{\partial_t \int_{y\cdot \vartheta \leq x \cdot \vartheta} \rho_t(y)dy}{\nu^\vartheta(T_{t,\vartheta}(x\cdot \vartheta))}\vartheta\,\cdot w_t \rho_t\,d\sigma(\vartheta)\,dx\\
        =&-\int \fint_{\mathbb{S}^{d-1}}  \frac{\int_{y\cdot \vartheta \leq x \cdot \vartheta}\text{div}(w_t\rho_t)\,dy}{\nu^\vartheta(T_{t,\vartheta}(x\cdot \vartheta))}\vartheta\,\cdot w_t \rho_t\,d\sigma(\vartheta)\,dx.
    \end{split}
    \end{equation*}
By Fubini, we change the order of integration and perform the integration on level sets $x\cdot \vart =c$:
\begin{equation}
\begin{split}
    \label{secondo termine}
        I_1=&-\fint_{\mathbb{S}^{d-1}} \int_{-\infty}^{+\infty} \frac{1}{\nu^\vartheta(T_{t,\vartheta}(c))} \left(\int_{y\cdot \vartheta \leq c}\text{div}(w_t\rho_t)\,dy\right) \left(\int_{x\cdot \vartheta=c}\vartheta\,\cdot w_t \rho_t\, \,dx\right)dc\,d\sigma(\vartheta)\\
        =&-\fint_{\mathbb{S}^{d-1}} \int_{-\infty}^{+\infty} \frac{1}{\nu^\vartheta(T_{t,\vartheta}(c))} \bigg(\int_{x\cdot \vartheta=c}\vartheta\,\cdot w_t \rho_t\,dx\bigg)^2\,d\sigma(\vartheta)\, dc\leq 0.
    \end{split}
\end{equation}
Combining \eqref{derivata seconda}, \eqref{primo termine} and \eqref{secondo termine}, we obtain
\begin{equation*}
    \begin{split}
        \frac{d^2}{dt^2}\F_\nu(\rho_t)\geq\left(\frac{1}{d}-L\right)W^2_2(\rho_0,\rho_1).
    \end{split}
\end{equation*}
\end{proof}
\section{Lipschitz bounds on $v[\rho]$}
\label{S: Lip bounds}
In this Section, we present two results, both giving bounds to the Lipschitz constant of $v$. First, we state and prove Lemma \ref{lemma 2}. This result shows that $\mathscr{F}_\nu$ is a $\lambda$-convex functional on the set of measures with bounded $L^p$ norm if $\nu$ has projections bounded away from zero. Then, we state Theorem \ref{Main theorem}, which instead proves $\lambda$-convexity on the set of measures with bounded $L^\infty$ norm, provided that $\nu$ satisfies a certain weaker condition, involving bounds on the density of $\nu$ itself and not of its projections. Indeed, we will see that lower bounds on the projections are very demanding: if $\nu$ is bounded from below and above on a bounded domain, most of the densities of its projections will vanish on the boundary of their support, because the measure of the section tends to $0$. Even more strikingly, the existence of measures with lower bounds on the densities of the projections is very restrictive, and is essentially a two- or three-dimensional affair. The precise conditions on $\nu$ will then be the object of study in the last part of this Section, where we will present some remarks and provide examples.
\subsection{Lipschitz bounds if the projections of $\nu$ are bounded from below}
\begin{lemma}
\label{lemma 2}
    Let $v[\rho]$ be defined by \eqref{velocity field}, $\nu \in \mathcal{P}_2(\R^d)$, and $K=\{\rho \in \mathscr{P}_2(\R^d)\colon \|\rho\|_{L^p}\leq M\}$, for some $p>d'$, $M>0$.
    Assume that, for almost every $\vartheta\in\mathbb S^{d-1}$, the support of $\nu^\vartheta$ is an interval $I_\vartheta$, and that $\nu^\vartheta\geq c_0>0$ almost everywhere on $I_\vartheta$.
    Then, there exists $L=L(c_0,M,p,d)$ such that, for any $\rho \in K$,
    $$\text{Lip}\left(v[\rho]+\frac{\mathrm{id}}{d}\right)\leq L,$$ and $\mathscr{F}_\nu$ is $\lambda$-geodesically convex on $K$ with $\lambda=\frac{1}{d}-L$.
\end{lemma}
\begin{proof}
    Deriving the relation (see Theorem \ref{brenier 1d})
        \begin{equation*}
        N_\vartheta(T_\vartheta(r))=R_\vartheta(r) \quad \text{for any }r\in \text{spt }\rho^\vart,
\end{equation*}
we deduce that
\begin{equation*}
    T'_\vart(r)=\frac{\rho^\vart(r)}{\nu^\vart(T_\vart(r))}.
\end{equation*}
Therefore,
    \begin{equation*}
    \begin{split}
        \left|Dv[\rho](x)+\frac{\mathrm{id}}{d}\right|&= \,\left|\fint_{\mathbb{S}^{d-1}} T_{\vartheta}'(x\cdot \vartheta)\vartheta \otimes \vartheta\right| \leq \fint_{\mathbb{S}^{d-1}} \bigg|T_{\vartheta}'(x\cdot \vartheta)\bigg|d\sigma(\vartheta)\\
        &\leq \fint_{\mathbb{S}^{d-1}} \frac{\rho^\vartheta(x\cdot \vartheta)}{\nu^\vart(T_{\vartheta}(x\cdot \vartheta))}\,d\sigma(\vartheta)\leq \frac{1}{c_0}\fint_{\mathbb{S}^{d-1}} \rho^\vartheta(x\cdot \vartheta)\,d\sigma(\vartheta)\\
        &=\frac{1}{c_0}\fint_{\mathbb{S}^{d-1}}\int_{e\cdot\vartheta=0}\int_0^{+\infty}\rho(x+er)r^{d-2}\,dr\,d\mathcal H^{d-2}(e)\,d\sigma(\vart)=\frac{(d-1)\omega_{d-1}}{d\omega_d c_0}\int \frac{\rho(y)}{|x-y|}\,dy,
    \end{split}
    \end{equation*}
where we used \cite[Lemma 5.1]{cozzi2025long} to obtain the last equality.
Moreover, if $\rho \in L^p$ for $p>d'$, we have
\begin{equation}
\label{rho/x stima Lp}
    \begin{split}
        &\int_{\R^d} \frac{\rho(y)}{|x-y|}\,dy\leq\int_{B_R(x)}\frac{\rho(y)}{|x-y|}\,dy+\int_{B_R^c(x)}\frac{\rho(y)}{|x-y|}\,dy\\
        &\leq \left(\int_{B_R}\frac{1}{|y|^{p'}}\,dy\right)^{\frac{1}{p'}}\left(\int_{\R^d}\rho^p\,dx\right)^{\frac{1}{p}}+\frac{1}{R} \leq c(d,p)R^{\frac{d}{p'}-1}\|\rho\|_{L^p}+\frac{1}{R}.
    \end{split}
\end{equation}
Optimizing over $R>0$, we choose $R=\|\rho\|_{L^p}^{-\frac{p'}{d}}(c(d,p)(-1+d/p'))^{-\frac{p'}{d}}$, thus obtaining
\begin{equation}
\label{rho/x stima Lp,2}
    \frac{(d-1)\omega_{d-1}}{d\omega_d c_0}\int_{\R^d}\frac{\rho(y)}{|x-y|}\,dy\leq \frac{(d-1)\omega_{d-1}M^{\frac{p'}{d}}}{d\omega_d c_0}\frac{d}{d-p'}
\left(\frac{c(d,p)(d-p')}{p'}\right)^{\frac{p'}{d}}=: L
\end{equation}
for any $\rho \in K$. Since $K$ is geodesically convex (see, for example \cite[Chapter 7]{santambrogio2015optimal}), we can apply Lemma \ref{Lemma 1} and deduce the claim.
\end{proof}
\subsection{Remarks about measures with projections bounded from below}
We collect here some remarks regarding measures $\nu\in \mathcal{P}_2(\R^d)$ satisfying the condition (appearing in Lemma \ref{lemma 2}) that $\nu^\vart(r)\geq c_0>0$ for a.e. $\theta \in \mathbb{S}^{d-1}$. We will also focus on the case where the $\nu^\vart$ are the uniform measure on a same interval.
\begin{remark}
    \label{proiezioni limitate}
    We do not expect that measures with the property that $\nu^\vart\geq c_0>0$ for a.e. $\vart \in \mathbb{S}^{d-1}$ exist in $\R^d$ for $d>3$. We will show in this remark that this is the case when we strengthen condition \eqref{eq: bounded projections} replacing it with the following property (which in turn implies \eqref{eq: bounded projections}): for almost every $\vart \in \mathbb{S}^{d-1}$ there exists an interval $I_\vart$ such that
    \begin{equation}
        \label{stronger condition bounded projections}
        \begin{split}
        \nu^\vart\geq c_0 \quad&\text{ on }I_\vart\\
        \nu^\vart=0 \quad &\text{ on }\R\setminus I_\vart \\
        C\coloneqq \bigcap_{\vart \in \mathbb{S}^{d-1}} \{x\cdot \vart \in I_\vart\} \quad &\text{ is uniformly convex}
        \end{split}
    \end{equation}
    Indeed, for any $\vart \in \mathbb{S}^{d-1}$, let $r_\vart=\sup_{x\in C}{x\cdot \vart}$ and $x_\vart \in \partial C$ be the boundary point such that $x_\vart\cdot \vart =r_\vart$. Then, consider $S_{\vart,\eps}=\{x \in C \colon |x\cdot\vart-r_\vart|\leq\eps\}$, which is a slice of size $\eps>0$ containing $x_\vart$ in its boundary. Since, by assumption, $C$ is uniformly convex and smooth (notice that without the assumption it would just be convex), given any $\eps>0$, for any ball $B_{\sqrt{\eps}}(x_\vart)$, it holds $S_{\vart,\eps}\subset B_{\sqrt{\eps}}(x_\vart)$. By \eqref{stronger condition bounded projections}, we have that $\nu(S_{\vart, \eps})\gtrsim \eps$, and hence $\nu(B_{\sqrt{\eps}}(x_\vart))\gtrsim \eps$. We now consider a collection of points $\{y_i\}\subset \partial C$ such that $|y_i-y_j|\geq 2\sqrt{\eps}$. The number of such points is of the order $(\sqrt{\eps})^{-(d-1)}$, since $\partial C$ is a $d-1$ dimensional surface. Define $(\partial C)_\eps\coloneqq \{x\in \R^d\colon \text{dist}(\partial C, x)\leq \eps\}$. Considering the above argument for any ball $B_{\sqrt{\eps}}(y_i)$ we obtain that
    \begin{equation*}
    \label{catena disuguaglianze epsilon arbitrario}1=\nu(C)\geq \nu((\partial C)_\eps)\geq \sum_{i=1}^{\eps^{-\frac{d-1}{2}}} \nu(B_{\sqrt{\eps}}(y_i))\gtrsim \sum_{i=1}^{\eps^{-\frac{d-1}{2}}} \eps\gtrsim \eps^{\frac{-d+3}{2}},\end{equation*}
    where we denoted by $\gtrsim$ inequalities holding up to multiplication by a dimensional constant. By arbitrariness of $\eps$, we deduce that it is necessary to assume $d\leq 3$. Moreover, when $d=3$, we also observe that $\nu$ must give positive mass to $\partial C$.
\end{remark}
\begin{remark}
\label{proiezioni uniformi}
The measure $\nu$ with uniform projections (i.e.$, \nu^\vartheta(r)=\chi_{[-1/2,1/2]}(r)$ for any $\vartheta\in \mathbb{S}^{d-1}$) satisfies \eqref{stronger condition bounded projections}: by Remark \ref{proiezioni limitate} such a measure only exists when $d=2,3$, and in this case the set $C$ mentioned there coincides with the closed ball of radius $1/2$ centered at the origin. More precisely, for $d=3$ the measure $\nu$ is the uniform probability measure on the sphere of radius $1/2$, and for $d=2$ it is an absolutely continuous measure with density $\nu(x)=\frac{1}{\pi\sqrt{1/4-|x|^2}}$ for $|x|<1/2$, and zero outside this ball.
One can check that in this case we have
\begin{equation*}
    \begin{split}
        \frac{SW_2^{2}(\rho,\nu)}{2}=\int \frac{|x|^2}{2d}\rho(x)\,dx-\frac{\omega_{d-1}}{2d\omega_{d}}\iint \rho(x)\rho(y)|x-y|\,dx\,dy+\frac{1}{24} \hspace{5mm} \text{for any }\rho \in \mathcal{P}_2(\R^d).
    \end{split}
\end{equation*}
For $\mathcal{V}(\rho)=\int V(x)\,d\rho(x)$ and $\mathcal{W}(\rho)=\iint W(x-y)\,d\rho(x)\,d\rho(y)$, we have that \begin{equation*}
    \frac{\delta \mathcal{V}}{\delta \rho}(\rho)=V(x),\quad \frac{\delta \mathcal{W}}{\delta\rho}(\rho)=2W\ast \rho.
\end{equation*}
 Hence, the velocity field of the associated SWF is
\begin{equation}
\begin{split}
\label{velocità proiezioni uniformi}
    v[\rho](x)=-\frac{x}{d}+\frac{\omega_{d-1}}{d\omega_{d}}\nabla W\ast \rho(x),
\end{split}
\end{equation}
where 
\begin{equation}
    \label{W}
    W(z)=|z|.
\end{equation}
\end{remark}
\subsection{Lipschitz bounds when $\nu$ satisfies Assumption \eqref{assumption nu}}
As usual, in the following result we set $\nu^\vartheta:=(\pi^\vartheta)_\#\nu$ (and identify it with its density) and denote by $N_\vartheta(r)$ the cumulative distribution of $\nu^\vartheta$, \textit{i. e.}, $N_\vartheta(r)=\int_{-\infty}^{r}\nu^\vartheta(s)\,ds$.
\begin{theorem} 
\label{Main theorem}
    Let $M>0$, $K=\{\rho \in \mathscr{P}_2(\R^d)\colon \|\rho\|_{L^\infty}\leq M\}$, and $\nu \in \mathscr{P}_2(\R^d)$.
    Assume that there exist $0<\alpha<(d-1)/d, \kappa>0$
    such that\\
\begin{equation}
\begin{split}
\label{assumption nu}
\nu^\vartheta(r)\geq  \kappa\left(\min \{N_\vartheta(r),1-N_\vartheta(r)\}\right)^\alpha
\hspace{3mm} &\text{for a.e. } \vartheta \in \mathbb{S}^{d-1}\text{ and every }r\in\R.
\end{split}
\end{equation}  
Then, $\mathscr{F}_\nu$ is $\lambda$-geodesically convex on $K$ with $$\lambda=\frac{1}{d}-\frac{c(d)}{\left(1-\frac{\alpha d}{d-1}\right)\kappa}M^{\frac{1}{d}},$$
where $c(d)$ is a constant depending only on $d$.
\end{theorem}
We postpone the proof of Theorem \ref{Main theorem} in Section \ref{sec: main theorem}. Rather, here we provide some remarks on condition \eqref{assumption nu}.

\begin{remark} If $\nu$ satisfies \eqref{assumption nu}, then $\nu$ must be compactly supported. 
Indeed, since $N_\vart'=\nu^\vart$, assumption \eqref{assumption nu} is equivalent to 
\begin{equation*}
    N'_\vart(r)\geq \kappa\min\{(1-N(r))^\alpha,N(r)^\alpha\}.
\end{equation*}
Moreover, since $\lim_{r\to +\infty} N(r)=1$, for $r$ large enough this leads to 
\begin{equation}
\label{equation not bounded solutions}
    N'_\vart(r)\geq \kappa(1- N_\vart(r))^\alpha.
\end{equation}
This in turn gives $(1-N_\vart)^{1-\alpha}\leq -c<0$, and it is not possible to have a nonnegative solution of this differential 
inequality on an unbounded domain for $\alpha <1$. 
\end{remark}
\begin{remark}
    We could relax assumption \eqref{assumption nu} by asking that $\nu^\vartheta(r)\geq \kappa g(\min\{N_\vartheta(r),1-N_\vartheta(r)\})$ where $g(t)$ satisfies the following properties: 
    \begin{enumerate}
        \item $g(ct)\geq c^\alpha g(t)$, for some $\alpha \in \R$,
        \item $\frac{1}{g\left(t^{\frac{d}{d-1}}\right)}\in L^1_{\text{loc}}$.
    \end{enumerate}
\end{remark}
\subsection{Examples of measures satisfying \eqref{assumption nu}}
\begin{example} (Projections bounded from below) Let $\nu$ be compactly supported with the property that $\text{spt }\,\nu^\vart$ is an interval $I_\vart$ for any $\vart \in \mathbb{S}^{d-1}$, and there exists $\kappa>0$ such that \begin{equation*}
    \label{projection bounded form below}
\nu^{\vartheta}(r)\geq c_0
\end{equation*}
for almost any $\vartheta \in \mathbb{S}^{d-1}$, $r\in I_\vart$.
Then one trivially has
\begin{equation}
    \label{eq: bounded projections}
    \nu_\vartheta(r)\geq c_0 N^{\alpha}_{\vartheta}(r),
\end{equation}
since $N_{\vartheta}(r)\leq 1$, for any $\alpha \in [0,(d-1)/d)$.
\end{example}
\begin{example}
\label{example prescribed behavior near boundary}(Prescribed behavior near the boundary of their convex support) The following Lemma, whose proof is presented at the end of this list, holds true
\begin{lemma}
\label{boundary of convex lemma}
Let $\nu\in \mathscr{P}_2(\R^d)$ be such that $\text{spt }(\nu)\subset C\subset \R^d$ where $C$ is a uniformly convex, smooth set. For any $r>0$, let $C_{r}\coloneqq \{x\in C \colon \text{dist}(x,\partial C)\leq r\}$. Assume there exist $r_0, c_1,c_2>0$ and $\beta_1\geq\beta_2>-1$ with 
\begin{equation}
\label{condition on exponents}
    0<\frac{\frac{d-1}{2}+\beta_1}{\frac{d+1}{2}+\beta_2}<\frac{d-1}{d},
\end{equation}
such that the restriction of $\nu$ to $C_{r_0}$ is absolutely continuous with respect to Lebesgue measure, with density satisfying 
\begin{equation}
\begin{split}
\label{inequality nu}
    c_1 \text{ dist}(x, \partial C)^{\,\beta_1}\leq\nu(x)\leq c_2\text{ dist}(x, \partial C)^{\,\beta_2 }\hspace{5mm} \text{for any }\hspace{1mm} x\in C_{r_0}.
\end{split}
\end{equation}
Then, $\nu$ satisfies \eqref{assumption nu}.
\end{lemma}
\end{example}
\begin{example} (Densities bounded on their convex support) We can take $\beta_1=\beta_2=0$ in Example \ref{example prescribed behavior near boundary}. Therefore, assumption \eqref{assumption nu} is satisfied by measures with support contained in a smooth and uniformly convex domain $C$ for which there exist $r_0>0$ and $0<c_1<c_2$ such that 
\begin{equation*}
    c_1\leq \nu(x)\leq c_2 \hspace{5mm} \text{for any }  x\in C_{r_0}.
\end{equation*}
\end{example}
\begin{proof}[Proof of Lemma \ref{boundary of convex lemma}]In the following, the symbols \say{$\gtrsim$} and \say{$\lesssim$} will denote inequalities holding up to constants only depending on the dimension, on the set $C$ and on $r_0$. 

Given any $\vartheta \in \mathbb{S}^{d-1}$, denote by $r^\vartheta_{-}=\inf_{t\in \R}\{\nu^\vartheta(t)>0\}$, $r^\vartheta_{+}=\sup_{t\in \R}\{\nu^\vartheta(t)>0\}$.

We distinguish three cases:
\begin{enumerate}
    \item $r\in (r^{\vart}_-+r_0,r^\vart_+-r_0)$,\\
    \item $r\in (r^{\vart}_-,r^\vart_-+r_0)$,\\
    \item $r\in (r^{\vart}_+-r_0,r^\vart_+)$.
\end{enumerate}
In each of these three domains, we will provide estimates from below for $\nu^\vartheta$, and estimates from above for $N^\vart$, in the order, thus finally deducing the desired comparison. 

\noindent\underline{Estimates from below for $\nu^\vartheta$}:
    \begin{enumerate}
        \item Up to reducing the value of $r_0$ if needed, for $r\in (r^{\vartheta}_{-}+r_0,r^{\vartheta}_{+}-r_0)$, the set $S_\vart(r,r_0):=\{x \in C \colon \, x\cdot\vart=r, \frac 12 r_0<\mathrm{dist}(x,\partial C)<r_0\}$ has $(d-1)-$measure which is at least $cr_0$. This provides,
        by \eqref{inequality nu}, \[
                \nu^\vartheta (r)=  \int_{x \in C \colon \, x\cdot\vart=r} \nu(x)\geq \int_{S_\vart(r,r_0)} \nu(x)\gtrsim r_0^{\beta_1+1}.
               \]

        \item For $r\in (r^\vartheta_{-},r^\vartheta_{-}+r_0)$ one has
            \begin{equation*}
                \nu^\vartheta(r)=\int_{x\in C \colon \, x\cdot \vart=r} \nu(x)= \int_{x\in C_{r_0} \colon \, x\cdot \vart=r} \nu(x)\geq c_1\int_{x\in C_{r_0}\colon \,x\cdot \vart =r} \text{dist}(\partial C,x)^{\beta_1}.
        \end{equation*}
        By the smoothness of $C$, for any $r\in (r^{\vart}_-,r^\vart_-+r_0)$, there exists a $(d-1)$ dimensional ball $B^{(d-1)}_R$ with radius $|r-r^\vart_-|^{1/2}\lesssim R\lesssim |r-r^\vart_-|^{1/2}$ such that $B^{(d-1)}_R\subset C_{r_0}\cap \{x\cdot \vart=r\} $. On every point of this ball we have $\text{dist}(\partial C,x)\leq |r-r^\vart_-|$, but if we restrict to a ball of radius $R/2$ we also have  $\text{dist}(\partial C,x)\gtrsim |r-r^\vart_-|$. We then obtain
        \begin{equation}
        \label{estimate from below (rvart-,r0+rvart-), beta1>0}
            \nu^\vart(r)\gtrsim c_1\int_{B^{(d-1)}_{R/2}}|r-r^\vart_-|^\beta_1 \gtrsim c|r-r^\vart_{-}|^{\beta_1+\frac{d-1}{2}}.
        \end{equation}
        \item On $(r^\vartheta_{+}-r_0,r^\vartheta_{+})$, we can repeat the same argument and obtain
            \begin{equation}
                \label{estimate from below (rvart+-r0,rvart+)}\nu^\vartheta(r)\gtrsim c_1|r-r^\vartheta_{+}|^{\beta_1+\frac{d-1}{2}}.
            \end{equation}
    \end{enumerate}
\noindent\underline{Estimates from above for $N_\vartheta$, $1-N_\vart$}:
    \begin{enumerate}
        \item For $r\in(r^\vartheta_{-}+r_0,r^\vartheta_{+}-r_0)$ it trivially holds
            \begin{equation*}
            N_\vartheta(r)\leq 1, \hspace{5mm} 1-N_\vartheta(r)\leq 1.
            \end{equation*}
        \item We start from an estimate from above on $\nu^\vart$. For $r\in(r_{-}^{\vartheta},r_{-}^{\vartheta}+r_0)$, we have that
        \begin{equation*}
            \nu^\vart (r) \leq c_2 \int_{x\in C\colon x\cdot \vart=r}\text{dist}(x,\partial C)^{\beta_2}\,dx,
        \end{equation*}
        where we used that for such an $r$ we have $C_{r_0}\cap \{x\cdot \vart=r\}=C\cap \{x\cdot \vart=r\}$.
        If $\beta_2\geq 0$, then one has that $\text{dist}(x,\partial C)\leq |r-r^\vart_-|$ for any $x\in C_{r_0}\cap \{x\cdot \vart=r\}$. Moreover, by the uniform convexity of $C$, for any $r\in (r^{\vart}_-,r^\vart_-+r_0)$, there exists a $(d-1)$ dimensional ball $B^{(d-1)}_R$ with radius $|r-r^\vart_-|^{1/2}\lesssim R\lesssim |r-r^\vart_-|^{1/2}$ such that $C\cap \{x\cdot \vart=r\}\subset B^{(d-1)}_R$. Thus
        \begin{equation}
        \label{estimate from above (rvart-,r0+rvart-), beta2>0}
            \nu^\vart (r) \leq c_2\int_{x\in C\colon x\cdot \vart=r} |r-r^\vart_-|^{\beta_2}\,dx\lesssim c_2 \int_{B_R^{(d-1)}}|r-r^\vart_-|^{\beta_2}\,dx\lesssim c_2 |r-r^\vart_-|^{\beta_2+\frac{d-1}{2}}.
        \end{equation}
        If instead, $\beta_2< 0$, we need a more refined argument, since upper bounds on the distance to $\partial C$ are not anymore what we need, and we cannot restrict to a smaller ball on which we can easily estimate from below such a distance, as we did when we were bounding $\nu^\vart$ from below. We note that the function $x\mapsto \text{dist}(x,\partial C)$ is concave (see Lemma \ref{subset of convex is convex} below). This function takes value $0$ on the relative boundary of $C\cap\{x\,:\,x\cdot\vart=r\}$. Moreover, there is a point $y_0\in C\cap\{x\,:\,x\cdot\vart=r\}$ such that $\text{dist}(y_0,\partial C)\gtrsim |r-r^\vart_-|$, that can be taken to be the projection onto the hyperplane $\{x\,:\,x\cdot\vart=r\}$ of the point $x_0\in \partial C$ realizing the infimum in $\inf_{x\in \partial C} x\cdot \vart$. This means, by concavity, that for every $x\in C\cap\{x\,:\,x\cdot\vart=r\}$ w have $\text{dist}(x,\partial C)\gtrsim |r-r^\vart_-|u(x)$, where the value $u(x)$ is defined as follows: let $y(x)$ be a point in the relative boundary of $C\cap\{x\,:\,x\cdot\vart=r\}$ such that the segment $[y_0,y(x)]$ passes through $x$, and define $u(x):=|x-y(x)|/|y_0-y(x)|$. 

        We then obtain 
$$          \nu^\vart (r) \lesssim \int_{x\in C\colon x\cdot \vart=r} |r-r^\vart_-|^{\beta_2}u(x)^{\beta_2}.$$
We then notice that we have $|\nabla u|\geq 1/R$, where $R$ is the radius of the ball $B^{d-1}_R(y_0)$ containing $C\cap \{x\,:\, x\cdot \vart=r\}$ and then
$$\int_{x\in C\colon x\cdot \vart=r} u(x)^{\beta_2}\leq R\int_{x\in C\colon x\cdot \vart=r} u(x)^{\beta_2}|\nabla u|=R\int_0^1 t^{\beta_2}\mathrm{Per}(\{u>t\})dt\lesssim R\cdot R^{d-2}\int_0^1 t^{\beta_2}dt.$$
Here we used first the coarea formula and then the fact that the sets $\{u>t\}$ are convex sets homothetical to $C\cap\{x\,:\,x\cdot\vart=r\}$, so that their perimeter is smaller than the one of the ball containing the latter set. Finally, we use $\beta_2>-1$ for the finiteness of the integral, and we then obtain
$$\nu^\vart (r) \leq  |r-r^\vart_-|^{\beta_2}R^{d-1}\simeq |r-r^\vart_-|^{\beta_2+\frac{d-1}{2}}.$$

        Using the previous estimates, we deduce
            \begin{equation*}
            \begin{split}
            N_\vartheta(r)&=\int_{r_{-}^\vartheta}^r \nu_\vartheta(s)\, ds\lesssim \int_{0}^{r-r^\vartheta_{-}} c_2\,s^{\frac{d-1}{2}+\beta_2}\,ds\leq \frac{c_2}{\frac{d+1}{2}+\beta_2}|r-r^{\vartheta}_{-}|^{\frac{d+1}{2}+\beta_2}.
            \end{split}
            \end{equation*}
        
        \item For $r \in (r^\vartheta_{+}-r_0,r^\vartheta_{+})$, one similarly has
            \begin{equation*}
                1-N_\vartheta(r)\lesssim \frac{c_2}{\frac{d+1}{2}+\beta_2}|r^\vartheta_{+}-r|^{\frac{d+1}{2}+\beta_2}.
            \end{equation*}
    \end{enumerate}
\noindent\underline{Comparing $\nu^\vartheta$ and $N_\vartheta$}:
    \begin{enumerate}
        \item On $(r^\vartheta_{-}+r_0,r^\vartheta_{+}-r_0)$, it holds  
            \begin{equation*}
            \nu^\vartheta(r)\gtrsim \frac{c_1}{\beta_1+1}\,r_0^{\beta_1+1}\gtrsim \frac{c_1}{\beta_1+1}\,r_0^{\beta_1+1}\min\{N_\vartheta(r),1-N_\vartheta(r)\}^\alpha,
            \end{equation*}
        for any $\alpha \in (0,(d-1)/d)$;
        \item On $(r^\vartheta_-,r^\vartheta_-+r_0)$, denoting by $\kappa$ a constant only depending on $\nu$, it holds
            \begin{equation*}
            \nu^\vartheta(r)\gtrsim \kappa N_\vartheta(r)^{\frac{\frac{d-1}{2}+\beta_1}{\frac{d+1}{2}+\beta_2}}= \kappa \min\{N_\vart (r),1-N_\vart(r)\}^{\frac{\frac{d-1}{2}+\beta_1}{\frac{d+1}{2}+\beta_2}}.
            \end{equation*}
        \item On $(r^\vartheta_{+}-r_0,r^\vartheta_+)$, one similarly has
            \begin{equation*}
            \nu^\vartheta(r)\gtrsim \kappa \min\{N_\vart (r),1-N_\vart(r)\}^{\frac{\frac{d-1}{2}+\beta_1}{\frac{d+1}{2}+\beta_2}}.
            \end{equation*}
    \end{enumerate}
Imposing condition \eqref{condition on exponents}, we obtain the claim.
\end{proof}
\begin{remark}
    From the proof of Lemma \ref{boundary of convex lemma} we deduce that, in the case when $\beta_1=\beta_2=\beta$, assumption \eqref{assumption nu} is satisfied as soon as $\max\left\{-1,-\frac{d-1}{2}\right\}<\beta<\frac{d-1}{2}$.
\end{remark}
 \begin{lemma}
        \label{subset of convex is convex}
            Let $K$ be convex, and let $R(x)=\text{dist}(x,\partial K)$. Then $R$ is concave on $K$.
        \end{lemma}
        \begin{proof}
            Given $x,y \in K$, we have $B(x,R(x)),B(y,R(y))\subset K$. Then, by convexity of $K$, we have $B(tx+(1-t)y,tR(x)+(1-t)R(y)) \subset K$, which shows $R(tx+(1-t)y)\geq tR(x)+(1-t)R(y) $, hence the concavity of $R$. 
        \end{proof}
       
\section{Proof of Theorem \ref{Main theorem}}
\label{sec: main theorem}
We will prove the following Theorem, which is a sufficient condition to obtain Theorem \ref{Main theorem}:
\begin{proposition}
\label{proposition assumption A1}
    Let $K=\{\rho \in \mathscr{P}_2(\R^d)\colon \|\rho\|_{L^\infty}\leq M\}$, for some $M>0$.
    If $\nu$ satisfies \eqref{assumption nu}, then there exists a constant $c(d)>0$ depending only on $d$, such that
    \begin{equation}
    \label{LipConstant}
    \text{Lip}\left(v[\rho]+\frac{\mathrm{id}}{d}\right)\leq \frac{c(d)}{\left(1-\frac{\alpha d}{d-1}\right)\kappa}M^{\frac{1}{d}}=:L
    \end{equation}
    for any $\rho \in K$. In particular, $\mathscr{F}_\nu$ is $\lambda$-geodesically convex on $K$ for $\lambda=\frac{1}{d}-L$.
\end{proposition}
\begin{proof}

Thanks to Lemma \ref{Lemma 1} and the fact that the set $K$ is geodesically convex, we just need to prove that \eqref{LipConstant} holds true for any $\rho \in K$ and any $\nu$ satisfying assumption \eqref{assumption nu}. By definition of $v[\rho,\nu]$, fixing a point $x \in \text{spt }\,\rho$, we have (the first inequality is obtained as in the proof of Lemma \ref{lemma 2})
\begin{equation*}
\label{lip velocity 1}
    \begin{split}
        \left|Dv[\rho](x)+\frac{\mathrm{id}}{d}\right|\,&\leq \fint_{\mathbb{S}^{d-1}}  \frac{\rho^\vartheta(x\cdot \vartheta)}{\nu^\vartheta(T_{\vartheta}(x\cdot\vartheta))}\,d\sigma(\vartheta)\\
        \leq& \frac{1}{\kappa}\fint_{\mathbb{S}^{d-1}}  \frac{\rho^\vartheta(x\cdot \vartheta)}{(\min \{N_\vartheta(T_{\vartheta}(x\cdot \vartheta)),1-N_\vartheta(T_{\vartheta}(x\cdot \vartheta))\})^\alpha}\,d\sigma(\vartheta)\\
        =& \frac{1}{\kappa}\fint_{\mathbb{S}^{d-1}}  \frac{\rho^\vartheta(x\cdot \vartheta)}{(\min \{R_\vartheta(x\cdot \vartheta),1-R_\vartheta(x\cdot \vartheta)\})^{\alpha}}\,d\sigma(\vartheta),
    \end{split}
\end{equation*}
where in the last equality, we used again the relation
\begin{equation*}
N_\vartheta(T_\vartheta(r))=R_\vartheta(r) \quad \text{for any }r\in \text{spt }\rho^\vart.
\end{equation*}
Denoting by $\chi^+(\vart)=\chi_{\{R_\vart(x\cdot \vart) \leq 1/2\}}(\vart)$ and $\chi^-(\vart)=\chi_{\{R_\vart(x\cdot \vart) \geq 1/2\}}(\vart)$, we obtain
\begin{equation}
\label{eq lip velocity 1}
    \begin{split}
        \left|Dv[\rho](x)+\frac{\mathrm{id}}{d}\right|\,\leq&\frac{1}{\kappa}\fint_{\mathbb{S}^{d-1}}\frac{\rho^\vart(x\cdot \vart)}{R_\vart(x\cdot\vart)^\alpha}\chi^+(\vart)\,d\sigma(\vart)+\frac{1}{\kappa}\fint_{\mathbb{S}^{d-1}}\frac{\rho^\vart(x\cdot \vart)}{(1-R_\vart(x\cdot\vart))^\alpha}\chi^-(\vart)\,d\sigma(\vart)\\
        \leq& \frac{1}{\kappa}\fint_{\mathbb{S}^{d-1}}\frac{\rho^\vart(x\cdot \vart)}{R_\vart(x\cdot\vart)^\alpha}\,d\sigma(\vart)+\frac{1}{\kappa}\fint_{\mathbb{S}^{d-1}}\frac{\rho^\vart(x\cdot \vart)}{(1-R_\vart(x\cdot\vart))^\alpha}\,d\sigma(\vart)\\
        =&\frac{2}{\kappa}\fint_{\mathbb{S}^{d-1}}\frac{\rho^\vart(x\cdot \vart)}{R_\vart(x\cdot\vart)^\alpha}\,d\sigma(\vart),
    \end{split}
\end{equation}
where the last equality was obtained by symmetry.
For any $e \in \mathbb{S}^{d-1}$, define
\begin{equation}
    \label{d-2 m(e) integral} 
    m_{x}(e)=\int_{0}^{+\infty}\rho(x+e\,r)r^{d-2}\,dr.
    \end{equation}
We have that, for any $\vartheta \in \mathbb{S}^{d-1}$
\begin{equation}
    \label{rho theta in x theta}
    \begin{split}
    \rho^\vartheta(x\cdot \vartheta)&=\int_{e\cdot\vartheta=0}m_x(e)\,d\mathcal H^{d-2}(e),\\ R_\vartheta(x\cdot \vartheta)&=\int_{e\cdot \vart\leq 0}\int_{0}^{+\infty}\rho(x+e\,r)r^{d-1}\,dr\,d\sigma(e).
    \end{split}
\end{equation}
Here and below, integrals over equators $\{e\in\mathbb S^{d-1}:e\cdot\vartheta=0\}$ and their subsets are taken with respect to $\mathcal H^{d-2}$, while integrals over hemispheres use $\sigma$. When $d=2$, $\mathcal H^0$ is counting measure.

Applying \eqref{rho theta in x theta} in \eqref{eq lip velocity 1}, we obtain
\begin{equation}
\label{eq: main thm}
        \left|Dv[\rho](x)+\frac{\mathrm{id}}{d}\right|\leq
        \frac{2}{\kappa}\fint_{\mathbb{S}^{d-1}}  \frac{\int_{e\cdot \vartheta=0}m_x(e)\,d\mathcal H^{d-2}(e)}{\left(\int_{e\cdot \vartheta \leq 0} \int_{0}^{+\infty}\rho(x+e\,r)r^{d-1}\,dr\,d\sigma(e)\right)^\alpha}\,d\sigma(\vartheta).
\end{equation}
Consider now
\begin{equation*}
    q(M,m_x(e)):=\inf_{0\leq\mu\leq M} \bigg\{\int_{0}^{+\infty}\mu(r)r^{d-1}\,dr \,\colon \int_{0}^{+\infty}\mu(r)r^{d-2}\,dr=m_{x}(e) \bigg\}.
\end{equation*}
The above infimum is achieved by $\mu(r)=M\chi_{(0,R)}(r)$. Here, necessarily,
\[R=\left(\frac{(d-1)m_{x}(e)}{M}\right)^{\frac{1}{d-1}},\]
since
\begin{equation*}
    \begin{split}
    m_{x}(e)=M\int_{0}^{+\infty} \chi_{(0,R)}(r)r^{d-2}\,dr=M \int_0^Rr^{d-2}\,dr=\frac{M}{d-1}R^{d-1}.
    \end{split}
\end{equation*}
Therefore,
\begin{equation}
    \label{optimization problem}
    \begin{split}
        q(M,m_x(e))=M^{-\frac{1}{d-1}}\frac{(d-1)^{\frac{d}{d-1}}}{d}m_x(e)^{\frac{d}{d-1}}.
    \end{split}
\end{equation}
Recalling \eqref{d-2 m(e) integral}, and thanks to \eqref{optimization problem}, we are able to estimate the denominator on the right hand side of \eqref{eq: main thm}, obtaining 
\begin{equation}
\label{Lip Est part 1}
    \begin{split}
        \left|Dv[\rho](x)+\frac{\mathrm{id}}{d}\right|\leq&\frac{2\,d^\alpha\,M^{\frac{\alpha}{d-1}}}{(d-1)^{\frac{\alpha d}{d-1}}\kappa}\fint_{\mathbb{S}^{d-1}}  \frac{\int_{e\cdot \vartheta =0}m_{x}(e)\,d\mathcal H^{d-2}(e)}{\left(\int_{e\cdot \vartheta \leq 0} m_x^{\frac{d}{d-1}}(e)\,d\sigma(e)\right)^{\alpha}}\,d\sigma(\vartheta)\\
        \leq&\frac{c(d)\,M^{\frac{\alpha}{d-1}}}{\kappa} \fint_{\mathbb{S}^{d-1}}  \frac{\int_{e\cdot \vartheta =0}m_{x}(e)\,d\mathcal H^{d-2}(e)}{\left(\int_{e\cdot \vartheta \leq 0} m_{x}(e)\,d\sigma(e)\right)^{\frac{d\alpha}{d-1}}}\,d\sigma(\vartheta),
    \end{split}
\end{equation}
where $c(d)>0$ denotes a constant depending only on $d$ (the term $(d-1)^{\frac{\alpha d}{d-1}}$ can indeed be bounded uniformly for $0\leq\alpha<(d-1)/d$).

We first present the proof in dimension $d=2$, for the reader's convenience. Next, we will generalize the same argument to any dimension.

\noindent\underline{Proof in dimension $d=2$}

Since in this case it holds
\begin{equation*}
    \fint_{\mathbb{S}^{d-1}}  \frac{\int_{e\cdot \vartheta =0}m_{x}(e)\,d\mathcal H^{d-2}(e)}{\left(\int_{e\cdot \vartheta \leq 0} m_{x}(e)\,d\sigma(e)\right)^{\frac{d\alpha}{d-1}}}\,d\sigma(\vartheta)=\fint_{\mathbb{S}^1}\frac{m_x(\vart+\pi/2)+m_x(\vart-\pi/2)}{\left(\int_{\vart-\pi/2}^{\vart+\pi/2}m_x(e)\,d\sigma(e)\right)^{2\alpha}}\,d\sigma(\vart),
\end{equation*}
the idea is to estimate this integral in the variable $\vart$ writing the numerator as a derivative of the denominator. Fix an orthonormal basis $(e_1,e_2)$ of $\R^2$. We parametrize $\mathbb{S}^1$ defining $\theta(e_1,e_2,t)=e_1\cos t+e_2\sin t$. Define $A(e_1,e_2,t)=\{e\in \mathbb{S}^{d-1}\colon e\cdot \theta (e_1,e_2,t)\leq 0\}$, and $A^+(e_1,e_2,t)=A(e_1,e_2,t)\,\cap \{e \in \mathbb{S}^{d-1}\colon e\cdot e_2\geq 0\}$. Define also
\begin{equation*}
    \varphi_{e_1,e_2}(t)= \int_{A^+(e_1,e_2,t)}m_x(e)\,d\sigma(e).
\end{equation*}
The function $\varphi$ is positive, decreasing for $t\in [0,\pi/2)\cup[3\pi/2,2\pi)$ and increasing in $[\pi/2,3\pi/2)$. Moreover it holds
\begin{equation}
\begin{split}
\label{AuxLemma in 2d}
    \theta(e_1,e_2,t)&=\theta(e_1,-e_2,-t),\\
    \varphi_{e_1,e_2}(t)&\leq \int_{A(e_1,e_2,t)}m_x(e)\,d\sigma(e),\\
    |\varphi'_{e_1,e_2}(t)|+|\varphi'_{e_1,-e_2}(-t)|&=\int_{e\cdot \theta(e_1,e_2,t)=0}m_x(e)\,d\mathcal H^{d-2}(e).
\end{split}
\end{equation}
Incidentally, let us note that two of the above identities have been obtained by exploiting the symmetries in our construction when considering an orthonormal basis with a different orientation with respect to the fixed one. This trick, even if more involved, will be the same at the core of the identities in Lemma \ref{AuxLemma}, needed for the construction in dimension $d\geq 3$.

We return to the proof in dimension $d=2$. By our definitions and \eqref{AuxLemma in 2d}, we have
\begin{equation}
\label{stima 2 2d}
\begin{split}
    &\int_{\mathbb{S}^{1}}  \frac{\int_{e\cdot \vartheta =0}m_{x}(e)\,d\mathcal H^{d-2}(e)}{\left(\int_{e\cdot \vartheta \leq 0} m_{x}(e)\,d\sigma(e)\right)^{2\alpha}}\,d\sigma(\vartheta)=\int_{0}^{2\pi}\frac{\int_{e\cdot \theta(e_1,e_2,t)=0}m_x(e)\,d\mathcal H^{d-2}(e)}{\left(\int_{A(e_1,e_2,t)}m_x(e)\,d\sigma(e)\right)^{2\alpha}}\,dt\\
    &=\int_{0}^{2\pi}\frac{|\varphi_{e_1,e_2}'(t)|}{\left(\int_{A(e_1,e_2,t)}m_x(e)\,d\sigma(e)\right)^{2\alpha}}\,dt+\int_{0}^{2\pi}\frac{|\varphi_{e_1,-e_2}'(-t)|}{\left(\int_{A(e_1,-e_2,-t)}m_x(e)\,d\sigma(e)\right)^{2\alpha}}\,dt\\
    &\leq\int_{0}^{2\pi}\frac{|\varphi_{e_1,e_2}'(t)|}{\left(\varphi_{e_1,e_2}(t)\right)^{2\alpha}}\,dt+\int_{0}^{2\pi}\frac{|\varphi_{e_1,-e_2}'(-t)|}{\left(\varphi_{e_1,-e_2}(-t)\right)^{2\alpha}}\,dt\\
    &\leq \frac{1}{1-2\alpha}TV\left[\varphi_{e_1,e_2}^{1-2\alpha}\right]+\frac{1}{1-2\alpha}TV\left[\varphi_{e_1,-e_2}^{1-2\alpha}\right],
    \end{split}
    \end{equation}
    Where $TV$ stands for the total variation. Using the piecewise monotone behavior of $t\mapsto \varphi_{e_1,e_2}(t)$, we have $TV\left[\varphi_{e_1,e_2}^{1-2\alpha}\right]\leq 2\left(\max_{t\in[0,2\pi)}\varphi_{e_1,e_2}(t)\right)^{1-2\alpha}$. Then, by \cite[Lemma 5.1]{cozzi2025long} it holds
\begin{equation}
\label{stima 3 2d}
    \max_{t\in[0,2\pi)}\varphi_{e_1,e_2}(t)\leq \int_{\mathbb{S}^{1}}\int_0^{+\infty}\rho(x+er)dr\,d\sigma(e)=\int_{\R^2}\frac{\rho(y)}{|x-y|}\,dy.
\end{equation}
Applying estimate \eqref{rho/x stima Lp,2} with $p=\infty$ to the right hand side of \eqref{stima 3 2d}, we deduce the existence of a constant $c$ such that 
\begin{equation}
\label{last 2d estimate}
    \max_{t\in[0,2\pi)}\varphi_{e_1,e_2}(t)\leq  c\,M^{\frac{1}{2}},
\end{equation}
 and the estimate does not depend on the chosen basis $(e_1,e_2)$. Combining \eqref{Lip Est part 1}, \eqref{stima 2 2d} and \eqref{last 2d estimate}, we conclude the proof of Proposition \ref{proposition assumption A1} when $d=2$.

\noindent\underline{Proof in dimension $d\geq 3$} 

Define $E$ to be the $(2d-3)-$dimensional set in $\R^{2d}$ given by couples $(e_1,e_2)\in\mathbb{S}^{d-1}\times \mathbb{S}^{d-1}$ such that $e_1\cdot e_2=0$, and endow it with the restriction of $\mathcal H^{2d-3}$. For any given triplet $(e_1,e_2,t)\in E\times [0,2\pi)$, define $\theta(e_1,e_2,t)=e_1\cos t+e_2\sin t$. For fixed $t$, the fibres of $(e_1,e_2)\mapsto\theta(e_1,e_2,t)$ carry $\mathcal H^{d-2}$; the fibres of the full map $\theta:E\times[0,2\pi)\to\mathbb S^{d-1}$ carry $\mathcal H^{d-1}$. All these measures are induced by the Euclidean product metric. As before, define \begin{equation*}
\begin{split}
A(e_1,e_2,t)&=\{e\in \mathbb{S}^{d-1}\colon e\cdot \theta (e_1,e_2,t)\leq 0\},\\
A^+(e_1,e_2,t)&=A(e_1,e_2,t)\,\cap \{e \in \mathbb{S}^{d-1}\colon e\cdot e_2\geq 0\},\\
\partial A(e_1,e_2,t)&=\{e\in \mathbb{S}^{d-1}\colon e\cdot \theta (e_1,e_2,t)= 0\},\\
\partial^+A(e_1,e_2,t)&=\partial A(e_1,e_2,t)\,\cap \{e \in \mathbb{S}^{d-1}\colon e\cdot e_2\geq 0\}.
\end{split}
\end{equation*}
Since no ambiguity can occur, we will simplify the notation, removing the dependence on $e_i,t$. We introduce again the function
\begin{equation*}
    \label{funzione_phi}
    \varphi_{e_1,e_2}(t)\coloneqq \int_{A^+} m_x(e)\, d\sigma(e).
\end{equation*}
Notice that $\varphi_{e_1,e_2}(t)$ is a positive function, decreasing for $t\in [0,\pi/2)\cup [3\pi/2,2\pi)$ and increasing in $[\pi/2,3\pi/2)$.
For the rest of the proof, we will make use of the following 
\begin{lemma} 
\label{AuxLemma}
Let $m_x$, $\varphi, \theta$ be defined as above, and $(e_1,e_2,t)\in E\times{[0,2\pi]}$. Then,
\begin{enumerate}
    \item There exists a function $h\colon \R^{2}\to \R$ such that, for any $t\in[0,2\pi)$, and for any $e_1,e_2\in E$, it holds
\begin{equation*}
\left|\int_{\partial^+A}m(e)
h(e\cdot e_1,e\cdot e_2)
\, d\mathcal H^{d-2}(e)\right|=
|\varphi'_{e_1,e_2}(t)|
\end{equation*}
    \item Fix $t\in [0,2\pi)$. The value of
\begin{equation*}
    \int_{\theta(e_1,e_2,t)=\vart,\, e\cdot e_2 \geq 0} h(e\cdot e_1,e\cdot e_2)\,d\mathcal H^{d-2}(e_1,e_2)
\end{equation*}
is constant with respect to $(e,\vartheta)\in E$.
    \item There exists a dimensional constant $c(d)$ such that, for any two given measurable functions $f,g$ it holds
\begin{equation*}
    \fint_{\mathbb{S}^{d-1}} \frac{\widetilde{f}(\vartheta)}{g(\vartheta)}\,d\sigma(\vartheta)=c(d)\int _{E\times[0,2\pi)}\frac{f(e_1,e_2,t)}{g(\theta(e_1,e_2,t))}\,d\mathcal H^{2d-3}(e_1,e_2)\,dt,
\end{equation*}
    where
\begin{equation*}
    \widetilde{f}(\vartheta)=\int_{\{\theta(e_1,e_2,t)=\vartheta\}}f(e_1,e_2,t)\,d\mathcal H^{d-1}(e_1,e_2,t).
\end{equation*}
\end{enumerate}
\end{lemma}
We postpone the proof of Lemma \ref{AuxLemma}. Combining the results therein contained, we can now conclude the proof of Proposition \ref{proposition assumption A1}. 
Fix $t\in [0,2\pi)$. Using the fact that, for any $(e_1,e_2,t)\in E\times [0,2\pi]$ it holds
\begin{equation*}
    \varphi_{e_1,e_2}(t)\leq \int_{A}m_x(e) \,d\sigma(e),
\end{equation*}
together with point $(1)$ of Lemma \ref{AuxLemma}, one obtains
\begin{equation*}
    \begin{split}
    \label{chain 1}    
    \int_{E}\frac{|\,\varphi'_{e_1,e_2}(t)|}{(\varphi_{e_1,e_2}(t))^{\frac{d\alpha}{d-1}}}\,d\mathcal H^{2d-3}(e_1,e_2)\geq & 
    \left|\int_{E}\frac{\int_{\partial^+A}m(e)h(e\cdot e_1,e\cdot e_2)\,d\mathcal H^{d-2}(e)}{\left(\int_{A}m(e)\,d\sigma(e)\right)^{\frac{d\alpha }{d-1}}}\,d\mathcal H^{2d-3}(e_1,e_2)\right|.
\end{split}
\end{equation*}
Integrating both sides with respect to $t \in [0,2\pi)$, we get
\begin{equation*}
    \int_0^{2\pi}\int_{E}\frac{|\,\varphi'_{e_1,e_2}(t)|}{(\varphi_{e_1,e_2}(t))^{\frac{d\alpha}{d-1}}}\,d\mathcal H^{2d-3}(e_1,e_2)\,dt\geq\left| \int_{E\times[0,2\pi)}\frac{\int_{\partial^+A}m(e)h(e\cdot e_1,e\cdot e_2)\,d\mathcal H^{d-2}(e)}{\left(\int_{A}m(e)\,d\sigma(e)\right)^{\frac{d\alpha }{d-1}}}\,d\mathcal H^{2d-3}(e_1,e_2)\,dt\right|.
\end{equation*}
Applying point $(3)$ of Lemma \ref{AuxLemma} to the right hand side of the above expression, we deduce
\begin{equation*}
\begin{split}
\int_0^{2\pi}\int_{E}&\frac{|\,\varphi'_{e_1,e_2}(t)|}{(\varphi_{e_1,e_2}(t))^{\frac{d\alpha}{d-1}}}\,d\mathcal H^{2d-3}(e_1,e_2)\,dt\\
\ge c(d)&\left|\fint_{\mathbb{S}^{d-1}}\int_{\{(e_1,e_2,t)\colon \theta(e_1,e_2,t)=\vartheta\}}\frac{\int_{e\cdot\vartheta=0, e\cdot e_2 \geq 0} m(e)h(e\cdot e_1, e\cdot e_2)\,d\mathcal H^{d-2}(e)}{\left(\int_{e\cdot \vartheta\leq0}m(e)\,d\sigma(e)\right)^{\frac{d\alpha}{d-1}}}\,d\mathcal H^{d-1}(e_1,e_2,t)\,d\sigma(\vart)\right|\\
    =c(d)&\left|\fint_{\mathbb{S}^{d-1}}\frac{\int_{e\cdot \vartheta=0}m_x(e)\int_{\{(e_1,e_2,t)\colon \theta(e_1,e_2,t)=\vartheta, e\cdot e_2 \geq 0\}}h(e\cdot e_1,e\cdot e_2)\, d\mathcal H^{d-1}(e_1,e_2,t)\, d\mathcal H^{d-2}(e)}{\left(\int_{e\cdot \vartheta\leq0}m(e)\,d\sigma(e)\right)^{\frac{d\alpha}{d-1}}}\,d\sigma(\vart)\right|.
    \end{split}
\end{equation*}
Finally, by point $(2)$ of Lemma \ref{AuxLemma}, we infer that
\begin{equation*}
    \begin{split}
        \fint_{\mathbb{S}^{d-1}}\frac{\int_{e\cdot \vartheta=0}m_x(e)\, d\mathcal H^{d-2}(e)}{\left(\int_{e\cdot \vartheta\leq0}m_x(e)\,d\sigma(e)\right)^{\frac{d\alpha}{d-1}}}\,d\sigma(\vart)&\leq c(d)\int_{E}\int_{0}^{2\pi}\frac{|\varphi'_{e_1,e_2}(t)|}{(\varphi_{e_1,e_2}(t))^{\frac{d\alpha}{d-1}}}\,dt\,d\mathcal H^{2d-3}(e_1,e_2)\\
        &\leq \frac{c(d)}{1-\frac{d\alpha}{d-1}}2\int_{E}\max_{t\in(0,2\pi)}\,\left|\varphi_{e_1,e_2}(t)\right|^{-\frac{d\alpha}{d-1}+1}\, d\mathcal H^{2d-3}(e_1,e_2).
    \end{split}
\end{equation*}
On the other hand, using again \cite[Lemma 5.1]{cozzi2025long} together with the estimate proved in \eqref{rho/x stima Lp,2} with $p=\infty$, we have 
\begin{equation*}
    \begin{split}
        \max_{t\in (0,2\pi)}\varphi_{e_1,e_2}(t)
        \leq& \int_{\mathbb{S}^{d-1}}\int_0^{+\infty}\rho(x+er)r^{d-2}\,dr\,d\sigma(e)
        =\int_{\R^d} \frac{\rho(y)}{|x-y|}\,dy\leq c_dM^{1/d}.
    \end{split}
\end{equation*}
Therefore, it holds
\begin{equation}
\begin{split}
    \label{m_x estimate}
    \int_{E} \max_{(0,2\pi)}|\varphi_{e_1,e_2}(t)|^{1-\frac{d\alpha}{d-1}}\,d\mathcal H^{2d-3}(e_1,e_2)\leq  \int_{E} \left| c(d)\,M^{\frac{1}{d}}\right|^{1-\frac{d\alpha}{d-1}}d\mathcal H^{2d-3}(e_1,e_2)\leq c(d)\,M^{\frac{1}{d}-\frac{\alpha}{d-1}}.
    \end{split}
\end{equation}
Finally, using estimate \eqref{m_x estimate} in \eqref{Lip Est part 1}, we obtain
\begin{equation*}
    \left|Dv[\rho](x)+\frac{\mathrm{id}}{d}\right|\leq \frac{c(d)}{(1-\frac{d\alpha}{d-1})\kappa}\,M^{\frac{1}{d}},
\end{equation*}
which concludes the proof of Proposition \ref{proposition assumption A1}.
\end{proof}
\begin{proof}[Proof of Lemma \ref{AuxLemma}]
Here we prove the results contained in Lemma \ref{AuxLemma}:
    \begin{enumerate}
        \item Fix $(e_1,e_2)\in E, t \in [0,2\pi)$. For any $e \in \mathbb{S}^{d-1}$, define $f_{e,e_1,e_2}(t)\coloneqq e\cdot \theta(e_1,e_2,t)$. We have that
        \begin{equation*}
        \begin{split}
        \frac{d}{dt}\varphi_{e_1,e_2}(t)&=\frac{d}{dt}\int_{\{e\cdot e_2\geq 0\}}m_x(e) \chi_{(-\infty,0)}\circ f_{e,e_1,e_2}(t)\,d\sigma(e)\\
        &=-\int_{\{e\cdot e_2 \geq 0\}}m_x(e)\delta_{0}(f_{e,e_1,e_2}(t))(-e\cdot e_1\sin t  +e\cdot e_2\cos t )\,d\sigma(e)\\
    \end{split}
    \end{equation*}
    Imposing $f_{e,e_1,e_2}(t)=0$ (in order to have a non-zero integrand), we can write $\cos t,\sin t$ in terms of $e\cdot e_1, e\cdot e_2$, obtaining
        \begin{equation*}
            \left|\varphi'_{e_1,e_2}(t)\right|=\left|\int_{e\cdot\theta(e_1,e_2,t)=0\cap e\cdot e_2\geq 0}m_x(e)h(e\cdot e_1,e\cdot e_2)\,d\mathcal H^{d-2}(e)\right|
        \end{equation*}
        for the function $h=-e\cdot e_1\sin t+e\cdot e_2 \cos t$. 
        \item 
        Fix $t\in[0,2\pi)$ and $(e',\vartheta')\neq (e,\vartheta)$, with $e\cdot\vart=e'\cdot\vart'=0$. One can always find $R\in O(d)$ such that $R e=e'$ and $R\vartheta=\vartheta'$. Indeed, there always exists an isometry sending a given orthonormal basis into another. Define now $e'_1=Re_1$ and $e'_2=Re_2$; we deduce that $\theta(e_1',e_2',t)=R\theta(e_1,e_2,t)=R\vartheta=\vartheta'$. Therefore
        \begin{equation*}
        \begin{split}
            \int_{\{(e_1,e_2)\colon \theta(e_1,e_2,t)=\vartheta', e'\cdot e_2\geq 0\}}&h(e'\cdot e_1, e'\cdot e_2)\,d\mathcal H^{d-2}(e_1,e_2)=\\&=\int_{\{(e_1,e_2)\colon \theta(e_1,e_2,t)=\vartheta', e'\cdot e_2\geq 0 \}}h(e\cdot R^{-1}e_1,e\cdot R^{-1}e_2)\,d\mathcal H^{d-2}(e_1,e_2)\\
            &=\int_{\{(R^{-1}e_1,R^{-1}e_2)\colon \theta(R^{-1}e_1,R^{-1}e_2,t)=\vartheta,  e \cdot R^{-1}e_2 \geq 0\}}h(e\cdot R^{-1}e_1,e\cdot R^{-1}e_2)\,d\mathcal H^{d-2}(e_1,e_2)\\
            &=\int_{\{(e_1,e_2)\colon \theta(e_1,e_2,t)=\vartheta , e\cdot e_2 \geq 0\}}h(e\cdot e_1, e\cdot e_2)\,d\mathcal H^{d-2}(e_1,e_2).
        \end{split}
        \end{equation*}
    We also notice that the integral in the lines above is not zero: since $e\cdot e_1 \cos t+e\cdot e_2 \sin t=0$, recalling $h=-e\cdot e_1 \sin t+e\cdot e_2 \cos t$,it follows that 
        \begin{equation*}
            h(e\cdot e_1, e\cdot e_2)=\text{sgn}(\cos t)\sqrt{(e\cdot e_1)^2+ (e\cdot e_2)^2}.
        \end{equation*}
    \item We will identify any point $p \in E\times [0,2\pi)$ with a point $(e_1,e_2,t)\in \R^{d}\times \R^d\times[0,2\pi)$. For any $(v,w,a)\in T_{p}(E\times[0,2\pi))$ and for some $\delta>0$, consider a curve \begin{equation*}
    \begin{split}
    \gamma \colon (-\delta,\delta)&\to \R^{d}\times \R^d \times \R\\
    \tau &\mapsto (\gamma_1(\tau),\gamma_2(\tau),\gamma_3(\tau)),
    \end{split}
    \end{equation*}
    with the property that $\gamma(0)=p\in E\times [0,2\pi)$, $\gamma_1'(0)=v,\gamma'_2(0)=w$ and $\gamma_3'(0)=a $.
    Without loss of generality, we can assume that $e_1=(1,0,\ldots,0)$ and $e_2=(0,1,0,\ldots,0)$. 
    In order to compute $D_{(e_1,e_2,t)}\,\theta\,[(v,w,a)]$, with $(v,w,a)\in T_{(e_1,e_2,t)}(E\times [0,2\pi))$, we write
    
\[
\begin{aligned}
D\theta_{(e_1,e_2,t)}(v,w,a)
&=v\cos t+w\sin t+a(-e_1\sin t+e_2\cos t).
\end{aligned}
\]
Thus, the differential is represented
in ambient coordinates by the matrix
\[
D\theta=\begin{pmatrix}
\cos t\,\mathrm{id} &
\sin t\,\mathrm{id} &
\zeta(t)
\end{pmatrix},
\]
where
\[
\zeta(t)=(-\sin t,\cos t,0,\ldots,0)^T.
\]
 Let $\eta= -e_1\sin t+e_2\cos t$, and complete $(e_1,e_2)$ to an orthonormal basis $(e_1,\ldots,e_d)$. A basis for $T_{(e_1,e_2,t)}M$ is given by the following $2d-2$ orthonormal vectors: for $j =3,\ldots,d$, consider
    \begin{equation*}
        \begin{split}
            \xi_1\coloneqq \frac{(e_2,-e_1,0)}{\sqrt{2}},\quad \xi_2\coloneqq (0,0,1),\\\xi_j\coloneqq (e_j,0,0),\quad \widetilde{\xi_j}\coloneqq (0,e_j,0).
        \end{split}
    \end{equation*}
    We have $D\theta[\xi_1]=\eta/\sqrt{2}$, $D\theta[\xi_2]=\eta$,$D\theta [\xi_j]=\cos t e_j$, $D\theta [\widetilde{\xi_j}]=\sin t e_j$. Therefore, with respect to the orthonormal basis $(\eta, e_3,\ldots,e_d)$ which spans the tangent space to $\mathbb{S}^{d-1}$, we deduce

\begin{equation*}
    D\theta D\theta^*=\text{diag}\left(\frac{3}{2},1,\ldots,1\right),
\end{equation*}
and hence $J_{d-1}\theta=\sqrt{3/2}$.
    By the coarea formula (see \cite{maggi2012sets}), we have that for any integrable measurable $\phi \colon E\times[0,2\pi)\to \R$, it holds
    \begin{equation*}
        \int_{E\times [0,2\pi)}J_{\theta}(p,t)\phi(p,t) d\mathcal H^{2d-3}(p)dt=\int_{\mathbb{S}^{d-1}}\left(\int_{\theta^{-1}(\vartheta)}\phi(p,t)\,d\mathcal H^{d-1}(p,t)\right)d\sigma(\vartheta).
    \end{equation*}
    Taking \[\phi(p,t)=\frac{f(p,t)}{g\circ \theta (p,t)}\] we obtain the claim.
    \end{enumerate}
    This concludes the proof of Lemma \ref{AuxLemma}
\end{proof}
\section{Uniqueness of the SWF}
\label{S: Uniqueness of the SWF}
In this section we furnish the proof of Corollary \ref{uniqueness corollary} using Theorem \ref{Main theorem}. We also provide an alternative proof under the assumption that the projections of the target measure $\nu$ are bounded from below. Finally, we show a third argument which applies if $\nu$ is the measure having uniform projections: we stress again that this can happen only if $d=2,3$ as a consequence of Remark \ref{proiezioni uniformi}.

We first recall the differentiation formulas for Wasserstein distances that will be used later. These results are discussed in \cite{santambrogio2015optimal}, Chapter 5.
\begin{lemma}
\label{derivative of distance}
    Let $p\in[1,+\infty)$ and let $(\rho_i(t))_{t\geq0}$, $i=1,2$, be locally absolutely continuous curves in time with values in $(\mathscr{P}_p(\R^d),W_p)$, consisting of probability measures absolutely continuous with respect to Lebesgue measure and solving
    \begin{equation*}
        \partial_t\rho_i+\nabla\cdot(\rho_i w_i)=0.
    \end{equation*}
    Then $t\mapsto W_p^p(\rho_1(t),\rho_2(t))$ is locally absolutely continuous, and the following identities hold for almost every $t$.

    If $\varphi=\varphi_t$ is a Kantorovich potential from $\rho_1(t)$ to $\rho_2(t)$ for the cost $\frac 1p|x-y|^p$, then we have 
\begin{equation}\label{diff1}
        \frac{1}{p}\frac{d}{dt}W_p^p(\rho_1(t),\rho_2(t)) =\int \nabla\varphi\cdot w_1\,d\rho_1(t)+\int\nabla\varphi^c\cdot w_2\,d\rho_2(t).
    \end{equation}
        This can be re-written with transport maps.
    If $p>1$ and $T=T_t$ is the optimal transport map from $\rho_1(t)$ to $\rho_2(t)$ for the same cost, then
    \begin{equation*}
        \frac{1}{p}\frac{d}{dt}W_p^p(\rho_1(t),\rho_2(t))
        =\int |x-T(x)|^{p-2}(x-T(x))\cdot\bigl(w_1(t,x)-w_2(t,T(x))\bigr)\,d\rho_1(t)(x),
    \end{equation*}
    where $|z|^{p-2}z$ is understood to be zero at $z=0$.

    If $p=1$, we choose $\varphi$ to be $1$-Lipschitz, so that $\varphi^c=-\varphi$; for any optimal transport map $T$ for the cost $|x-y|$, we have
    \begin{equation*}
        \frac{d}{dt}W_1(\rho_1(t),\rho_2(t))
        =\int\nabla\varphi\cdot(w_1-w_2\circ T)\,d\rho_1(t).
    \end{equation*}
\end{lemma}

We now go on with some proofs of uniqueness for the Sliced Wasserstein Flow.

\begin{proof}[Proof of Corollary \ref{uniqueness corollary}]
   The existence of a solution is proved in \cite{bonnotte2013unidimensional} together with local (in time) $L^\infty$ bounds. 
   Let $(\rho_i(t))_{t\geq 0}$ for $i=1,2$ be two solutions of the sliced Wasserstein flow \eqref{SWF eq rho_0}, and let $T>0$. Let $M$ be a constant such that $\rho_i(t,x)\leq M$ for a.e. $x$ and every $t\in [0,T]$ and $i=1,2$. 
   By Theorem \ref{Main theorem} we know that, under assumption \eqref{assumption nu} $\mathscr{F}_\nu$ is $\lambda$-geodesically convex on $K=\{\rho\,:\, \rho\leq M\}$ for some $\lambda \in \R$. We then use formula \eqref{diff1} and a well-known estimate due to geodesic convexity (see \cite{SantSurvey}, Lemma 4.12), to obtain
   \begin{equation*}
       \frac{d}{dt}W^2_2(\rho_1(t),\rho_2(t))\leq -2\lambda W^2_2(\rho_1(t),\rho_2(t)),
   \end{equation*}. Uniqueness then follows by Gronwall inequality
   which proves $\rho_1(t)=\rho_2(t)$ if $\rho_1(0)=\rho_2(0)$ (see \cite{ambrosio2005gradient}).
\end{proof}

\begin{remark}
 We provide here an alternative proof of the uniqueness of the SWF,
which does not use Theorem \ref{Main theorem}, by computing the
variation of the distance $W_1$ between two solutions.
We assume that, for almost every $\vartheta\in\mathbb{S}^{d-1}$,
the support of $\nu^\vartheta$ is an interval $I_\vartheta$,
and that there exists $\kappa>0$, independent of $\vartheta$,
such that $\nu^\vartheta\geq \kappa$ almost everywhere on
$I_\vartheta$. The assumption on $\nu$ implies that it is compactly supported, and this uniqueness result also concerns compactly supported solutions. Let $\rho_1, \rho_2$ be compactly supported solutions of \eqref{SWF eq rho_0}. We will denote by $T,\varphi$ the optimal transport map from $\rho_1$ to $\rho_2$ and the associated Kantorovich potential, respectively. Moreover, we will denote by $T_i$ the optimal transport map from $\rho_i$ to $\nu$, and by $T_i^\vartheta$ the one from $\rho_i^\vartheta$ to $\nu^\vartheta$. By Lemma \ref{derivative of distance}, for almost every $t$,
    \begin{equation*}
    \begin{split}
        \partial_t W_1(\rho_1(t),\rho_2(t))=&\int \varphi(\partial_t \rho_1-\partial_t \rho_2)=\int \nabla \varphi\cdot(v[\rho_1]\rho_1-v[\rho_2]\rho_2)\\
        =&\int \nabla \varphi\cdot(v[\rho_1]-v[\rho_2]\circ T)\rho_1,
    \end{split}
    \end{equation*}
    Then, recalling that $\text{Lip}(\varphi)\leq 1$,
    \begin{equation}
    \label{remark W1 first}
        \begin{split}
            \partial_t W_1(\rho_1(t),\rho_2(t))=&\int \nabla \varphi\,\cdot \left(v[\rho_1]-v[\rho_2]\right)d\rho_1+\int \nabla\varphi\, \cdot \left(v[\rho_2]-v[\rho_2]\circ T\right)d\rho_1\\
            \leq& \int \fint_{\mathbb{S}^{d-1}} \left| T_1^{\vartheta}(x\cdot \vartheta)-T_2^\vartheta(x\cdot \vartheta)\right|d\rho_1(x)\,d\sigma(\vart)+\text{Lip}(v[\rho_2])\int|T(x)-x|d\rho_1.
        \end{split}
    \end{equation}
    Thanks to the assumption $\nu^\vart\geq \kappa$ (which in turn implies that $N^{[-1]}_\vart$ is $ \kappa^{-1}-$Lipschitz continuous for every $\vart$), we have
    \begin{equation}
    \label{remark W1 second}
    \begin{split}
        \left|T_{\vart,1}(x\cdot \vart)-T_{\vart,2}(x\cdot \vart)\right|&=\left|N_\vart^{[-1]}(R_{1,\vart}(x\cdot \vart))-N_\vart^{[-1]}(R_{2,\vart}(x\cdot \vart))\right|\\
        &\leq \frac{1}{\kappa}  \left|R_{1,\vart}(x\cdot \vart)-R_{2,\vart}(x\cdot \vart)\right|.\\
        \end{split}
    \end{equation}
    We recall that, for any $\mu,\gamma \in \mathscr{P}_2(\R)$ it holds
    \begin{equation}
    \label{W_1 dimensione 1}
        W_1(\mu,\gamma)=\int_0^1|R^{-1}_{\mu}(t)-R^{-1}_{\gamma}(t)|dt= \int_{\R}|R_\mu(t)-R_\gamma(t)|\,dt.   
    \end{equation}
    Moreover, since $\rho_1 \in L^\infty$ and it is compactly supported, we can provide an upper bound to $\rho_1^\vart$.
     Combining \eqref{remark W1 first}, \eqref{remark W1 second} and \eqref{W_1 dimensione 1}, together with the fact that $\rho^\vart_1\in L^\infty$, we obtain
    \begin{equation*}
    \begin{split}
        \partial_t W_1(\rho_1(t),\rho_2(t))\leq&\frac{1}{\kappa}\fint_{\mathbb{S}^{d-1}} \int_{\R} |R_{1,\vart}(r)-R_{2,\vart}(r)|\,d\rho_1^{\vart}(r)\,d\sigma(\vart)+\text{Lip}(v)W_1(\rho_1(t),\rho_2(t))\\
        \leq &\frac{C}{\kappa}W_1(\rho_1(t),\rho_2(t))
    \end{split}
    \end{equation*}
for some constant $C>0$ that is fixed once $T>0$ is. By Gronwall lemma, we deduce again the uniqueness of SWF.
\end{remark}



\begin{remark}
As shown in \cite[Section 4.1]{cozzi2025long}, if in \eqref{SWF eq rho_0} both the initial datum $\rho_0$ and the target measure $\nu$ have bounded support, then the solution $\rho_t$ also has this property for any $t\geq 0$. Under this assumption, and if the target measure $\nu$ has uniform projections (namely, $\nu^{\vartheta}(r)=\chi_{[-1/2,1/2]}(r)$ for any $\vartheta\in \mathbb{S}^{d-1}$), then the uniqueness of solutions to \eqref{SWF eq rho_0} can be proved following a strategy presented in \cite{lagoutiere2024vanishing}. Indeed, repeating verbatim the computations made in the proof of \cite[Theorem 4.1]{lagoutiere2024vanishing}, for any $\rho_1,\rho_2$ compactly supported solutions of the SWF with target measure $\nu$, and for $p\geq 1$, the following estimate holds. Here we fix a finite time interval $[0,T]$ and choose $R>0$ so that the supports of both solutions lie in $B_R$ throughout this interval; $\text{Lip}_{B_R}$ denotes the Lipschitz constant on this ball.
\begin{equation*}
    \frac{1}{p}\frac{d}{dt}W_p^p(\rho_1(t),\rho_2(t))
    \leq \underbrace{\text{Lip}_{B_R}(v[\rho_1])\int|T_t(x)-x|^p\,d\rho_1}_{\coloneqq I_1}+\underbrace{W_p(\rho_1,\rho_2)^{p-1}\left(\int|v[\rho_1]-v[\rho_2]|^p(T_t(x))\,d\rho_1\right)^{1/p}}_{\coloneqq I_2}.
\end{equation*}
Then, for $I_1$, recalling $W(z)=|z|$ and setting $a_d=\omega_{d-1}/(d\omega_d)$, we have
\begin{equation*}
\begin{split}
    \text{Lip}_{B_R}(v[\rho_1(t)])
    &\leq \frac1d+a_d\|D^2W*\rho_1(t)\|_{L^\infty(B_R)}\\
    &\leq \frac1d+a_d\|D^2W\|_{L^{p_0}(B_{2R})}\|\rho_1(t)\|_{L^{p_0'}}.
\end{split}
\end{equation*}
On the other hand, for $I_2$, following again the argument presented in \cite{lagoutiere2024vanishing}, we have that for any $q\geq 1$ it holds
\begin{equation*}
    I_2 \leq a_d W^p_p(\rho_1,\rho_2)\|D^2 W\|_{L^{pq'}(B_{2R})}\|\rho_2(t)\|^{\frac{1}{p}}_{L^q}.
\end{equation*}
Choose $1\leq p<p_0<d$ and set $q=p_0/(p_0-p)$, so that $pq'=p_0$. We deduce
\begin{equation*}
\begin{split}
    \frac1p\frac{d}{dt}W_p^p(\rho_1(t),\rho_2(t))
    &\leq I_1+I_2\\
    &\leq W_p^p(\rho_1(t),\rho_2(t))\left[\frac1d+a_d\|D^2W\|_{L^{p_0}(B_{2R})}\left(\|\rho_1(t)\|_{L^{p_0'}}+\|\rho_2(t)\|_{L^q}^{1/p}\right)\right].
\end{split}
\end{equation*}
By \eqref{W} we observe that $|D^2W|\lesssim |z|^{-1}\in L^{p_0}_{\text{loc}}$ for any $p_0<d$, hence $p_0'>d'$.
Applying again the Gronwall lemma, and thanks to Theorem \ref{stime Bonnotte Lp}, we can conclude that $\rho_1=\rho_2$.
\end{remark}

\section*{Acknowledgements}

This work was supported by the European Union via the ERC AdG 101054420 EYAWKAJKOS project and by the ARQUS alliance, through the support of a collaborative project between the teams in Lyon and Padua.

\section*{AI statement}

In this work, LLM models have been used to compute some exact integral values throughout the proof, and to proofread the final manuscript.

\printbibliography
\end{document}